\documentclass[12pt]{article}

\usepackage{amssymb, latexsym, amsmath, amsthm,hyperref}

\newcommand{\subref}[2][]{\text{\tiny #1\ref{#2}}}

\makeatletter
\newtheorem*{rep@theorem}{\rep@title}
\newcommand{\newreptheorem}[2]{%
\newenvironment{rep#1}[1]{%
 \def\rep@title{#2 \ref{##1}}%
 \begin{rep@theorem}}%
 {\end{rep@theorem}}}
\makeatother

\newtheorem{theorem}{Theorem}
\newreptheorem{theorem}{Theorem}
\newtheorem{lemma}[theorem]{Lemma}
\newtheorem{corollary}[theorem]{Corollary}
\newtheorem{conjecture}[theorem]{Conjecture}
\newtheorem{problem}[theorem]{Problem}

\newcommand{\bw}{\operatorname{bw}} 
\newcommand{\eps}{\varepsilon}

\begin{document}
\title{Ramsey-goodness---and otherwise}
\author{Peter Allen\footnote{Instituto de Matem\'atica e Estat\'{\i}stica, Universidade de
    S\~ao Paulo, Rua do Mat\~ao 1010, 05508--090~S\~ao Paulo, Brazil.
    Email:~\texttt{allen@ime.usp.br}. During this work PA was supported
    by DIMAP and Mathematics Institute, University of Warwick, U.K., EPSRC
    award EP/D063191/1.}
    , Graham Brightwell\footnotemark[2] and Jozef Skokan\footnotemark[2]  }
\maketitle
\footnotetext[2]{Department of Mathematics, London School of Economics,
    Houghton Street, London WC2A 2AE, United Kingdom. Email:~\texttt{\{g.brightwell|j.skokan\}@lse.ac.uk}. }
    
\begin{abstract}
A celebrated result of Chv\'atal, R\"odl, Szemer\'edi and Trotter states (in slightly weakened form) that, 
for every natural number $\Delta$, there is a constant $r_\Delta$ such that, for any connected 
$n$-vertex graph $G$ with maximum degree $\Delta$, the Ramsey number $R(G,G)$ is at most $r_\Delta n$, provided $n$ is sufficiently large.

In 1987, Burr made a strong conjecture implying that one may take $r_\Delta = \Delta$. 
However, Graham, R\"odl and Ruci\'nski showed, by taking $G$ to be a suitable expander graph, 
that necessarily $r_\Delta > 2^{c\Delta}$ for some constant $c>0$. We show that the use of 
expanders is essential: if we impose the additional restriction that the bandwidth of $G$ be 
at most some function $\beta (n) = o(n)$, then
$R(G,G) \le (2\chi(G)+4)n\leq (2\Delta+6)n$, i.e., $r_\Delta = 2\Delta +6$ suffices. On the 
other hand, we show that Burr's conjecture itself fails even for $P_n^k$, the $k$th power of 
a path $P_n$.




Brandt showed that for any $c$, if $\Delta$ is sufficiently large, there
are connected $n$-vertex graphs $G$ with $\Delta(G)\leq\Delta$ but
$R(G,K_3)>cn$. We show that, given $\Delta$ and $H$, there are
$\beta>0$ and $n_0$ such that, if $G$ is a connected graph on $n\ge n_0$ vertices with maximum degree at most
$\Delta$ and bandwidth at most $\beta n$, then we
have $R(G,H)=(\chi(H)-1)(n-1)+\sigma(H)$, where $\sigma(H)$ is the smallest
size of any part in any $\chi(H)$-partition of $H$.  We also show that the same conclusion holds without 
any restriction on the maximum degree of $G$ if the bandwidth of $G$ is at most $\eps(H) \log n/\log\log n$.  

\end{abstract}

\thispagestyle{empty}
\section{Introduction}

Given two graphs $G$ and $H$, the Ramsey number $R(G,H)$ is defined to be the smallest $N$ such that,
however the edges of $K_N$ are coloured with red and blue, there exists either a red copy of $G$ or a
blue copy of $H$.

In the 1980s, Burr~\cite{BurrConj} and Burr and Erd\H{o}s~\cite{BurrErdos} made various seemingly natural
conjectures on the magnitudes of Ramsey numbers $R(G,H)$ in which one or both graphs is sparse.  In the
1990s, Brandt~\cite{Brandt} and Graham, R\"odl and Ruci\'nski~\cite{GRRLinRam} used expander graphs to give
counterexamples to these conjectures.  Our aim in this paper is to show that limiting the expansion of the
graphs suffices to (almost) rescue the conjectures.

There are two slightly different sets of results within this paper.  We are interested in
$R(G,H)$ in the case when $H$ is a (small) fixed graph, and $G$ may be much larger, and we are also
interested in the case when $H=G$.  The results we prove have a similar flavour, and we use similar
techniques.  To start with, we think of $H$ as a fixed graph.

A very simple general lower bound on the Ramsey number, given by Chv\'atal and Harary~\cite{ChvHar}, is
$R(G,H)\geq (\chi(H)-1)(|G|-1)+1$ for connected graphs $G$ -- here $|G|$ denotes the number of vertices
of $G$.  To see this, consider a two-colouring of the complete graph consisting of
$\chi(H)-1$ disjoint red cliques each on $|G|-1$ vertices, with only blue edges between them.  The red
components are too small to contain $G$, and the chromatic number of the subgraph of blue edges is too
small for $H$.

Burr and Erd\H{o}s~\cite{BurrErdos} defined a connected graph $G$ to be \textit{$p$-good} if
$R(G,K_p)=(p-1)(|G|-1)+1$; in other words, if the Ramsey number is equal to the lower bound of
Chv\'atal and Harary.  A family $\mathcal{G}$ of graphs is defined to be \textit{$p$-good} if there is some
$n_0$ such that every $G\in\mathcal{G}$ with $|G|\geq n_0$ is $p$-good.  Burr and Erd\H{o}s were interested
in the problem of finding families of graphs which are $p$-good for all $p$.  Chv\'atal~\cite{Chv} showed
that the family of trees is $p$-good for all $p$.  Burr and Erd\H{o}s~\cite{BurrErdos} showed that for any
$k$, the family of connected graphs with bandwidth at most $k$ is $p$-good for all $p$ (although the value of
$n_0$ does increase with $p$).  They made several conjectures regarding larger families of $p$-good graphs,
many of which have been answered in a recent paper of Nikiforov and Rousseau~\cite{NikiRous}.  One remaining
open question is to determine whether the family of hypercubes is $p$-good for any $p\geq 3$.

As observed by Burr~\cite{Burr}, the idea of the construction of Chv\'atal and Harary can be adapted to
give a stronger lower bound in many cases.  To explain this, we define a graph parameter: for any graph
$H$ of chromatic number $\chi(H)$, let
$\sigma(H)$ be the minimum size of a colour class in a proper $\chi(H)$-colouring of $H$.

\begin{lemma} [Burr~\cite{Burr}] \label{LowBd} For all graphs $G$ and $H$, with $G$ connected and
$|G| = n > \sigma(H)$, we have
$$
R(G,H)\geq (\chi(H)-1)(n-1)+\sigma(H).
$$
\end{lemma}
\begin{proof}
Let $K$ be the two-coloured complete graph whose red edges form a disjoint union of $\chi(H)-1$ cliques
on $n-1$ vertices and one on $\sigma(H)-1$ vertices.  Then the red components of $K$ are again too small
to contain $G$.  If there were an embedding of $H$ into the blue edges then it would give a $\chi(H)$-colouring
of $H$ with one class of size at most $\sigma(H)-1$, which contradicts the definition of $\sigma(H)$.
\end{proof}

We say that a connected graph $G$ is \textit{$H$-good} if $R(G,H)=(\chi(H)-1)(|G|-1)+\sigma(H)$, and that
a family of graphs $\mathcal{G}$ is \textit{$H$-good} if all sufficiently large members of $\mathcal{G}$ are
$H$-good.  Finally, we call a graph class $\mathcal{G}$ \textit{always-good} if $\mathcal{G}$ is $H$-good
for every graph $H$.  Burr~\cite{Burr} showed that, for all graphs $G_1$, the class of graphs homeomorphic
to $G_1$ is always-good.

 We mention two barriers to always-goodness. One necessary property for 
 a family ${\mathcal G}$ to be always-good is that ${\mathcal G}$ does 
 not contain arbitrarily large graphs $G$ in which the maximum degree 
 $\Delta(G)$ is nearly as large as $|G|$.
 An explicit version of this principle is illustrated by a construction 
 of Brown~\cite{Brown} yielding, for every prime $p$, a 
 $(p^2+p+1)$-vertex graph $H_p$ with minimum degree $p+1$ containing no 
 copy of $K_{2,2}$.  Let $\Gamma$ be the two-coloured complete graph 
 obtained from $H_p$ by colouring its edges blue and non-edges red. By 
 definition, $\Gamma$ does not contain either a blue copy of $K_{2,2}$ 
 or a vertex of red-degree $p^2$.  It follows that, if $G$ is any graph
 on $p^2+p$ vertices with $\Delta(G)\geq p^2$, then $R(G,K_{2,2})\ge 
 p^2+p+2$, which is strictly greater than 
 $(\chi(K_{2,2})-1)(|G|-1)+\sigma(K_{2,2}) = p^2+p+1$, so $G$ is not
 $K_{2,2}$-good.  One can clearly obtain better bounds by using larger 
 bipartite graphs in place of $K_{2,2}$.

 The class of trees is $K_p$-good for all $p$ but -- for instance by 
 the above argument -- not always-good. The graph families considered 
 by Nikiforov and Rousseau also contain graphs with such high degrees, 
 and are thus not always-good.

 A second barrier to always-goodness is strong vertex expansion. Burr 
 and Erd\H{o}s conjectured that, for any $\Delta$ and $p$, if $n$ is 
 sufficiently large, then any $n$-vertex graph $G$ with $\Delta(G)\leq 
 \Delta$ is $p$-good; Burr~\cite{BurrConj} made the natural 
 strengthening to conjecture that for any $\Delta$, the graph class 
 $\{G\colon\Delta(G)\leq \Delta\}$ is always-good. However
 Brandt~\cite{Brandt} showed that, for $\Delta \ge 168$, the family of 
 all $\Delta$-regular graphs is not even $K_3$-good; Nikiforov and 
 Rousseau~\cite{NikiRous} reduced this degree requirement to $100$. 
 Both proofs relied upon the fact that such graphs can have strong 
 vertex expansion properties. To be precise, Brandt proved the 
 following result showing that Burr and Erd\H{o}s' conjecture is
 already wrong by an arbitrarily large factor for $p=3$.

 \begin{theorem} [Brandt~\cite{Brandt}] Let $c$ be any constant. Then 
  if $\Delta$ and $n$ are sufficiently large, there exists an 
  $n$-vertex graph $G$ with $\Delta(G)\leq \Delta$ such that 
  $R(G,K_3)>cn$.
 \end{theorem}

 We show that Brandt's use of expander graphs is necessary: if 
 $\mathcal{G}$ is a graph class with not only bounded maximum
 degree but also suitably limited expansion, then the more general
 conjecture of Burr is rescued.

 We first state our results in terms of restricting the bandwidth 
 of~$G$.

 Given a graph $F$, the $k$th power of $F$, denoted $F^k$, is the graph 
 with vertex set $V(F)$ and edges between any two vertices whose 
 distance in $F$ is at most $k$.  In particular, $P_n^k$ is the $k$th 
 power of the $n$-vertex path $P_n$. For any graph $G$ on $n$ vertices, 
 the \textit{bandwidth} of $G$, $\bw(G)$, is the smallest $k$ such that 
 $G$ is a subgraph of $P_n^k$.

 First we consider what happens if we bound the bandwidth of graphs $G$ 
 in the class ${\mathcal G}$, but do not further bound the degree. In 
 this case, to show that ${\mathcal G}$ is always-good, it suffices to 
 show that the class of graphs $P_n^k$ is always-good.  One may think 
 of $k$ as being fixed, but in fact our proof works provided $k$ grows 
 more slowly than $\log n / \log \log n$.

\begin{theorem}\label{2alwaysgood}

 For each fixed graph $H$ and natural number $k$, 
 $R(P_n^k,H) = (\chi(H)-1)(n-1) + \sigma(H)$ whenever 
 $n \ge (20k|H|)^{16k|H|}$.

 In particular, if $\kappa(n)$ is any function with 
 $\kappa(n) = o(\log n/\log \log n)$, then the graph class  
 $\mathcal{B}_\kappa=\{G : \bw(G)\leq \kappa(|G|)$ and $G$ is 
 connected$\}$ is always-good.
\end{theorem}

 This result, even in the case where $\kappa(n)$ is the constant  
 function $k$, includes the result of Burr and Erd\H os stating that 
 ${\mathcal B}_k$ is $p$-good for all $p$, as well as the result of 
 Burr that the class of graphs homeomorphic to any fixed $G_1$ is 
 always-good.

 If $\kappa(n) = n^\eps$, for any fixed $\eps > 0$, then the class 
 ${\mathcal B}_\kappa$ is not always-good. To see this, note that 
 $R(K_s,K_t) = \Omega(s^{t/2})$ for fixed $t$ as $s \to \infty$, by a 
 standard probabilistic argument \cite{SpLB}. Therefore also 
 $R(P_n^s,K_t)=\Omega(s^{t/2})$, since $P_n^s$ contains $K_s$, and so 
 the class of connected graphs with bandwidth at most $n^\eps$ is not 
 $K_{4/\eps}$-good.

 Our proof of Theorem~\ref{2alwaysgood} uses a method  
 from~\cite{TwoCycle}, inspired by the Szemer\'edi Regularity 
 Lemma~\cite{SzReg}. This method yields a partition of $V(G)$ and an 
 auxiliary graph $G^*$ on the parts, but the partition arises
 from the direct use of Ramsey's theorem rather than an iterated 
 refinement procedure, enabling us to obtain a somewhat reasonable 
 bound on the size of $n$ we need.

 Our next theorem shows that, if we put an absolute bound on the 
 maximum degree of graphs in our class, it is enough to impose any 
 upper bound on the bandwidth that is sublinear in the order of the 
 graph.

\begin{theorem} \label{NExpAlwaysGood} 
 For every fixed $\Delta$, and every function $\beta(n)=o(n)$, the 
 graph class
 $$
  \mathcal{G}_{\Delta,\beta} 
   = 
  \{ G : \Delta(G) \le \Delta, \bw(G) \le \beta(|G|), 
          \mbox{ and $G$ is connected}\}
 $$
 is always-good.
\end{theorem}

In other words, a class of connected graphs is always-good if the maximum degree of graphs in the class is
bounded and, for any $\beta >0$, all sufficiently large graphs $G$ in the class have bandwidth at most $\beta |G|$.

Our proof of Theorem~\ref{NExpAlwaysGood} follows the same lines as 
Theorem~\ref{2alwaysgood}, using also an embedding method of B\"ottcher, Schacht and 
Taraz~\cite{BST}, which does involve the use of the Regularity Lemma.


As we now explain, Theorem~\ref{NExpAlwaysGood} can be converted to a result where
the expansion properties of the graph $G$ are explicitly limited.

B\"ottcher, Pruessmann, Taraz and W\"urfl~\cite{BPTW} define a
graph $G$ to be $(b,\varepsilon)$-\textit{bounded} if, for every subgraph $G'$
of $G$ with $|G'|\geq b$, there exists a set $U\subset V(G')$ with $|U|\leq |G'|/2$ and
$|\Gamma(U)-U|\leq\varepsilon|U|$.  Here $\Gamma(U)$ denotes the neighbourhood of $U$ in the
graph $G'$.  They proved the following theorem.

\begin{theorem}  [B\"ottcher, Pruessmann, Taraz and W\"urfl~\cite{BPTW}] \label{BPTW}
For any $\Delta\ge 1$ and $\beta_1>0$, there exist $\varepsilon>0$,
$\beta_2>0$ and $n_0$ such that, whenever $n\geq n_0$, every $(\beta_2 n,\varepsilon)$-bounded
$n$-vertex graph $G$ with $\Delta(G)=\Delta$ has $\bw(G)\leq \beta_1 n$.
\end{theorem}

 We call a graph class $\mathcal{G}$ \textit{non-expanding on large 
 subsets} if for any $\beta,\eps>0$ the following is true. There exists 
 $n_0$ such that if $G\in\mathcal{G}$ has $n\geq n_0$ vertices, then 
 $G$ is $(\beta n,\eps)$-bounded.

 An immediate corollary of Theorem~\ref{NExpAlwaysGood}, together with
 Theorem~\ref{BPTW}, is the following.

\begin{corollary}\label{NExpCor} 
 Given $\Delta$, let $\mathcal{G}$ be a class of connected graphs of 
 maximum degree $\Delta$ which is non-expanding on large subsets. Then 
 $\mathcal{G}$ is always-good.
\end{corollary}

 Corollary~\ref{NExpCor} is best possible in the following sense.  
 Brandt's method~\cite{Brandt} can be adapted easily to show that, for 
 any sufficiently large $\Delta$ and $\beta>0$, if $n$ and $p$ are 
 sufficiently large, and $G$ is an $n$-vertex graph with 
 $\Delta(G)\leq \Delta$ which does possess a subgraph $G'$ on at least 
 $\beta n$ vertices with strong expansion properties (for example: if 
 $G'$ is a typical $\Delta$-regular graph), then $G$ is not $p$-good.

There is another sense in which Brandt's family of counterexamples is the simplest possible. 
He showed that the class of connected graphs with maximum degree at most $\Delta$ is not $H$-good for 
$H=K_3$.  On the other hand, 
this class of graphs is $H$-good for every \textit{bipartite} $H$. 

\begin{theorem} \label{bipartite} 
For each fixed $\Delta$, let ${\mathcal D}_\Delta$ be the class of connected graphs with maximum
degree at most $\Delta$.  Then ${\mathcal D}_\Delta$ is $H$-good for every bipartite graph $H$. 
\end{theorem} 
Theorem \ref{bipartite} is a consequence of the following result of Burr, Erd\H os, Faudree, Rousseau and Schelp \cite{Burretal}.
\begin{theorem}
 For each bipartite graph $H$, there exist $\epsilon>0$ and $n_0$ such that, if $n\geq n_0$ and $G$
 is an $n$-vertex connected graph with maximum degree at most $\epsilon n^{1/(\sigma(H)+2)}$, then 
 $G$ is $H$-good. 
\end{theorem} 
For convenience, we present a short proof of Theorem \ref{bipartite} later.
\bigskip

 We note that there is a natural extension of the notion of 
 $H$-goodness to the multicolour setting. For graphs $H_1, \dots, H_r$, 
 we say that a connected graph $G$ is \textit{$(H_1,\ldots,H_r)$-good} 
 when there are integers $W$ and $Z$ (depending on $(H_1,\ldots,H_r)$ but
 not on $G$) such that $R(G,H_1,\ldots,H_r)=W(|G|-1)+Z$.  We say that a graph
 class $\mathcal{G}$ is \textit{multicolour-always-good} when, for every 
 $r\geq 2$ and every collection of graphs
 $H_1,\ldots,H_r$, $G$ is $(H_1,\ldots,H_r)$-good for all sufficiently 
 large $G\in\mathcal{G}$. In Section~\ref{Mcol}, we discuss the 
 problems of finding $W$ and $Z$, and prove the following theorem.

\begin{theorem}\label{MAGisAG} 
 If $\mathcal{G}$ is any always-good class of graphs, then 
 $\mathcal{G}$ is multicolour-always-good.
\end{theorem}

\medskip

We now turn our attention to the case where $G=H$, and $H$ is again of bounded maximum degree.
Burr~\cite{BurrConj} conjectured that, for each fixed $\Delta$, if $H$ is a sufficiently large connected
graph with maximum degree at most $\Delta$, then $H$ is itself $H$-good, i.e.,
$$
R(H,H) = (\chi(H)-1)(|H|-1) + \sigma(H). 
$$
In his paper, Burr warns that this conjecture ``may be too bold'', and indeed so it proved.

Burr's conjecture would imply that, for each fixed $\Delta$, $R(H,H) \le \Delta |H|$, whenever $H$ is a
sufficiently large graph with maximum degree $\Delta$.  Chv\'atal, R\"odl, Szemer\'edi and Trotter~\cite{CRST}
proved that some result along these lines is true: for every $\Delta$, there is some constant $r$ such that,
whenever $H$ has maximum degree $\Delta$, $R(H,H) \le r |H|$.

For each fixed $\Delta$, let $r_\Delta = \displaystyle\liminf_{n\to \infty} \max\{ R(H,H)/n : H$ is a connected
graph on $n$ vertices with maximum degree at most $\Delta \}$.  So the result of Chv\'atal, R\"odl,
Szemer\'edi and Trotter is that $r_\Delta$ is finite for all $\Delta$, and Burr's conjecture would imply
that $r_\Delta \le \Delta$.

The question of determining the rate of growth of $r_\Delta$ was addressed by
Graham, R\"odl and Ruci\'nski~\cite{GRRLinRam}, who proved the following theorem, giving bounds in
both directions.

\begin{theorem} [Graham, R\"odl and Ruci\'nski~\cite{GRRLinRam}]  \label{GRR}
There exist constants $c,c'>0$ such that the following hold.
\begin{itemize}
\item [(i)] Whenever $H$ is an $n$-vertex graph with $\Delta(H)\leq\Delta$, $R(H,H)\leq 2^{c'\Delta\log^2\Delta}n$.
\item [(ii)] For each sufficiently large $n$, there exists a bipartite $n$-vertex graph $H$ with $\Delta(H)\leq\Delta$ and
$R(H,H)> 2^{c\Delta}n$.
\end{itemize}
\end{theorem}

Theorem~\ref{GRR} implies that $2^{c\Delta} \le r_\Delta \le 2^{c'\Delta \log^2\Delta}$, and in particular
that Burr's conjecture is false.  The proof of the lower bound in Theorem~\ref{GRR} relies upon a
(probabilistic) construction of a graph $H$ with maximum degree $\Delta$ and good expansion properties.

Recently, Fox and Sudakov~\cite{FoxSud} established the alternative upper bound $R(H,H) \le 2^{c\Delta(H) \chi(H)}|H|$, for some 
explicit constant $c$. In particular, if $H$ is bipartite, this matches the form of the lower 
bound in Theorem~\ref{GRR}. The result for bipartite graphs was obtained independently
by Conlon \cite{Conlon}, and very recently Conlon, Fox and
Sudakov~\cite{ConFoxSud}, improving on Theorem~\ref{GRR}, showed that there is a
constant $c''$ such that $r_\Delta\le 2^{c''\Delta\log\Delta}$. 

We show that the use of expansion in the lower bound is necessary -- that is,
when both maximum degree and expansion are restricted, the Ramsey number may be
bounded above by a function linear in both $n$ and $\Delta$.  In fact, we will
prove something slightly stronger: when expansion is appropriately restricted,
the Ramsey number is primarily controlled by the chromatic number of $H$, not
the maximum degree, as in Burr's conjecture.

Observe that simply requiring $H$ to fail some global expansion condition will
not suffice to bound $R(H,H)$ below $2^{c\Delta}n$. To see this, take some
large $\Delta$ and $n$, let $H'$ be an $(n/10)$-vertex graph with
$\Delta(H')\leq \Delta$ and $R(H',H')> 2^{c\Delta}n/10$, and form $H$ by
adding $9n/10$ isolated vertices to $H'$.  The new graph $H$ is a poor
expander, yet $R(H,H) > 2^{c\Delta-4}n$. It follows that, as before, we need
to restrict the expansion of all large subgraphs of $H$, or the bandwidth of $H$.

We shall show that, if the degree of $H$ is at most $\Delta$, $H$ is sufficiently large, and the
bandwidth of $H$ is at most $\beta |H|$ for some small constant $\beta$, then $R(H,H) \le (2\chi(H) + 4)|H|$.
Thus imposing a restriction on the bandwidth of $H$ almost rescues Burr's conjecture.

Our first task in this direction is to investigate the Ramsey numbers of powers of paths.

In Section~\ref{SecPnkCnk}, we consider the Ramsey numbers $R(P_n^k,P_n^k)$ and $R(C_n^k,C_n^k)$.
Gerencs\'er and Gy\'arf\'as~\cite{GerGy} showed that $R(P_n,P_n)=n-1+\sigma(P_n)$, and (for $n\geq 5$)
Faudree and Schelp~\cite{FauSch} and Rosta~\cite{Rosta} showed that $R(C_n,C_n)=(\chi(C_n)-1)(n-1)+\sigma(C_n)$,
matching the lower bounds in Lemma~\ref{LowBd}, and in Burr's conjecture.
It is natural to ask whether this continues to hold (for sufficiently large $n$) for each $k$: for powers of
paths, this would mean that $R(P_n^k,P_n^k) = (k + \frac{1}{k+1})n + O(1)$.

In Section~\ref{SecPnkCnk}, we give a construction showing that this is not the case.  For
convenience, we state the result when $n$ is a multiple of $k+1$.

\begin{theorem}\label{LowerBound} For $k \ge 2$, and $n$ a multiple of $k+1$, we have
$$
R(C^k_n,C^k_n),\,R(P^k_n,P^k_n)\geq (k+1)n-2k.
$$
\end{theorem}

This shows in particular that even bounding the bandwidth of $H$ by a constant does not suffice to rescue
Burr's conjecture.

We suspect that the inequality above is tight, at least for powers of paths.  We have not been
able to show this, but we offer the following upper bounds, which differ from the lower bounds
by a multiplicative factor slightly greater than~2.

\begin{theorem}\label{RPnkCnk} 
 For any $k\geq 2$, we have
 \begin{align*}
  R(P_n^k,P_n^k) &\leq
   \left(2k+2+\frac{2}{k+1}\right)n+o(n)\,, \\
  \intertext{and}
  R(C_n^k,C_n^k) &\leq
   \left(2\chi(C_n^k)+\frac{2}{\chi(C_n^k)}\right)n+o(n)\,.
 \end{align*}
\end{theorem}

Using Theorem~\ref{RPnkCnk}, together with the embedding method of B\"ottcher, Schacht and Taraz~\cite{BST},
we prove the following result.

\begin{theorem}~\label{BBwRam} Given $\Delta\ge 1$, there exist $n_0$ and $\beta_1$ such that, whenever
$n\geq n_0$ and $H$ is an $n$-vertex graph with maximum degree at most $\Delta$ and
$\bw(H)\leq \beta_1 n$, we have $R(H,H)\leq (2\chi(H)+4)n$.
\end{theorem}

As before, we can use Theorem~\ref{BPTW} to convert the hypothesis of sublinear bandwidth to a
condition on the expansion of all large subgraphs.

\begin{corollary} For any $\Delta \ge 1$, there exist $n_0$, $\beta_2$ and $\varepsilon$ such that,
whenever $n\geq n_0$ and $H$ is a $(\beta_2 n,\varepsilon)$-bounded $n$-vertex graph with
maximum degree at most $\Delta$, we have $R(H,H)\leq (2\chi(H)+4)n$.
\end{corollary}

One might hope to show that, under the conditions of Theorem~\ref{BBwRam} or its corollary,
$R(H,H) \le (\chi(H) + C)n$.  In order to prove this, a first step would be to show such a
bound for the case $H=P_n^k$, but there are likely to be additional difficulties in the
general case.
One asymptotically sharp result in this direction has been proved.

\begin{theorem}[S\'ark\"ozy, Schacht and Taraz~\cite{SarSchTar}] For every
$\gamma>0$ and $\Delta$, there exist $\beta>0$ and $n_0$ such that, whenever 
$n\geq n_0$ and $H$ is an $n$-vertex bipartite graph with maximum degree at most
$\Delta$, $\bw(H)\leq \beta n$, and parts of size $t_1$ and $t_2$ (where
$t_1\le t_2$), we have
\[R(H,H)\le (1+\gamma)\max(2t_1+t_2,2t_2)\,.\]
\end{theorem}

A final observation is that combining the Four Colour
Theorem~\cite{AppHak,AppHakKoch} and another result of B\"ottcher, Pruessmann,
Taraz and W\"urfl~\cite{BPTW}, namely that the bandwith of every $n$-vertex planar graph
of maximum degree $\Delta$ is bounded by $15n/\log_{\Delta}n$, we obtain, as a corollary 
to Theorem~\ref{BBwRam}, the following.

\begin{corollary}\label{PlanarCor} For every $\Delta$ there exists $n_0$ such
that, whenever $n\geq n_0$ and $H$ is an $n$-vertex planar graph with maximum degree $\Delta$,
we have $R(H,H)\le 12n$.
\end{corollary}


\section{A version of the blow-up lemma}

In our proofs, we need an embedding lemma, similar in style to the Blow-up Lemma of
Koml\'os, S\'ark\"ozy and Szemer\'edi~\cite{Blowup}.  That result could be used as it stands,
but using an alternative approach allows us to obtain significantly better bounds on the sizes
of the graphs to which our results apply.

Instead of considering `$(\varepsilon,\delta)$-super-regular' pairs of sets (as in the original
Blow-up Lemma), where there are relatively few but well distributed edges, we will be interested in
pairs of vertex sets within two-coloured complete graphs which do not contain a red $K_{s,s}$ for
some $s$.  By the K\"{o}vari-S\'{o}s-Tur\'{a}n theorem~\cite{KovSosTur}, this condition strongly limits
the number and distribution of red edges. We give two forms.

\begin{theorem} [K\"{o}v\'ari, S\'{o}s and Tur\'{a}n~\cite{KovSosTur}] \label{KSTBound}
\begin{itemize}
\item[]
\item [(a)] For all $s, n \in {\mathbb N}$ with $n \ge s^2$, any $n$-vertex graph which does not contain $K_{s,s}$ has at most $2n^{2-\frac{1}{s}}$ edges.

\item [(b)] Let $G$ be a bipartite graph, with parts $X$ and $Y$, which does not contain a
copy of $K_{s,s}$. If $2\big(s/|Y|\big)^\frac{1}{s}\leq p\leq 1$, then at most $2s/p$ vertices in $X$ have degree greater than $p|Y|$.
\end{itemize}
\end{theorem}

%
%

We can now state and prove our embedding lemma.

\begin{lemma} \label{MyBlowUp}
Let $t$, $s$, $r$ and $d$ be natural numbers, with $s \ge d^2$.
Suppose that the edges of a complete graph $G$ with vertex set $V$ are coloured red and blue.
Let $V_1,\ldots,V_t$ be disjoint subsets of $V$, each of size at most $s$.
Define a graph $G'$ on disjoint vertex sets $V'_1,\ldots,V'_t$, where
$|V'_i|=\max\big(|V_i|-\big\lfloor 4r^2s^\frac{2r-1}{2r}(d+1)\big\rfloor,0\big)$ for each $i$, by putting
edges between all vertices in $V'_i$ and $V'_j$ whenever there is no red $K_{r,r}$ between
$V_i$ and $V_j$ in $G$. If $H$ is any subgraph of $G'$ with maximum degree $d$, then
$G$ contains a blue copy of $H$.
\end{lemma}

\begin{proof}
Let $G^{\rm blue}$ be the spanning subgraph of $G$ whose edges are the blue edges of $G$.

If $r=1$, then $G'$ is isomorphic to a subgraph of $G^{\rm blue}$, and the result is trivially true. We will assume from now on that $r\geq 2$.


Let $p=4r^2s^{-1/2r}$: then, for each $i$, $|V_i|-|V'_i|\leq p(d+1)s$.
Note that, if $p\geq \frac{1}{d+1}$, then each set $V'_i$ is empty and 
there is nothing to prove, so we can assume $p<\frac{1}{d+1}$. By 
Theorem~\ref{KSTBound}(b), if $X$ and $Y$ are vertex sets within a pair 
$(V_i,V_j)$ that does not contain a red $K_{r,r}$, and $|Y|\geq \frac{\sqrt{s}}{2^r r^{2r-1}}$, then at most $2r/p$ vertices in $X$ 
have red-degree greater than $p|Y|$.

Choose an embedding $\psi:V(H)\rightarrow V(G')$. Let $V(H) = \{x_1, 
x_2,\ldots\}$.  We will successively choose vertices $\phi(x_1), 
\phi(x_2),\ldots\in V(G)=V(G^{\rm blue})$ which give an embedding $
\phi$ of $H$ into $G^{\rm blue}$. For each $x_i\in H$, set 
$A_{x_i,1}=V_j$, where $V'_j$ is the part of $G'$ containing 
$\psi(x_i)$.

The set $A_{x_i,t}$ is called the \textit{allowed set} of $x_i$ at time 
$t$; we invariably choose $\phi(x_t)$ to be within its allowed set at 
time $t$. We maintain two properties. First, if $x_ix_j\in E(H)$ and 
$x_i$ has been embedded, then the allowed set of $x_j$ is entirely 
within the blue-neighbourhood of $x_i$.  Second, if, at time $t$, $x_i$ 
has not yet been embedded, then its allowed set has size larger than 
$ps/2=2r^2s^{\frac{2r-1}{2r}}$: this quantity is definitely larger than 
the $\frac{\sqrt{s}}{2^r r^{2r-1}}$ required to apply 
Theorem~\ref{KSTBound}. At time $1$, the first condition is trivially 
satisfied, and the second is true by the choice of the sizes of the 
$V'_i$.

At time $t$, we choose a vertex $\phi(x_t)\in A_{x_t,t}$ which is 
blue-adjacent to at least $(1-p)|A_{x_\ell,t}|$ of the vertices of 
$A_{x_\ell,t}$ for each $\ell>t$ with $x_\ell$ adjacent to $x_t$. This 
is possible since, by Theorem~\ref{KSTBound}(b), for each of the at 
most $d$ neighbours of $x_t$ not yet embedded, at most $2r/p$ 
vertices in $A_{x_t,t}$ fail to be blue-adjacent to 
$(1-p)|A_{x_\ell,t}|$ of the vertices of $A_{x_\ell,t}$, and 
$|A_{x_t,t}|\geq ps/2>d\frac{2r}{p}$ by the choice of $s$.

Having chosen $\phi(x_t)$, for each $\ell>t$ we set $A_{x_\ell,t+1}$ 
equal to $A_{x_\ell,t}-\{\phi(x_t)\}$ if $x_tx_\ell\notin E(H)$, and 
equal to $A_{x_\ell,t}\cap\Gamma_{\rm blue}(\phi(x_t))$ if $x_\ell$ is
adjacent to $x_t$.  It is clear that the allowed sets maintain the 
first property.  If $x_i$ is a vertex not yet embedded, with 
$\psi(x_i)\in V'_j$, then there are two reasons why a vertex $v\in V_j$ 
should not be in $A_{x_i,t+1}$: first, it might not be blue-adjacent to 
one of the at most $d$ embedded neighbours of $x_i$, and second, it 
might be the image under $\phi$ of some preceding vertex (in $V'_j$) of 
$H$. Thus we have
\[
 |A_{x_i,t+1}|
  \geq
 (1-p)^d|V_j|-|V'_j|
  >
 (1-pd)|V_j|-(|V_j|-p(d+1)s)
  \geq 
 \frac{ps}{2},
\]
so that the allowed sets maintain both the required conditions. It 
follows that this algorithm successfully embeds $H$ into $G^{\rm blue}$.
\end{proof}

\section{Powers of paths versus general graphs}\label{2col}

The aim of this section is to prove Theorem~\ref{2alwaysgood}, stating that $R(P_n^k,H)=(\chi(H)-1)(n-1)+\sigma(H)$
whenever $n$ is sufficiently large in terms of $k$ and $|H|$, and therefore that the family ${\mathcal B}_k$ of graphs
of bandwidth at most~$k$ is always-good.

First we need to give a stability version of an old theorem of Erd\H{o}s~\cite{ErdLow} stating that
$R(P_n,K_\ell)=(\ell-1)(n-1)+1$.

\begin{lemma} \label{StabLem2}
Given $\ell\geq 2$, $0\leq\alpha<1/2$, $0\leq\eps<(1-\alpha)/\ell$, and 
$n\geq 1 + 1/\eps$, the following is true. If $G$ is a two-coloured 
graph on $(\ell-1-\alpha)(n-1)$ vertices, in which every vertex is 
adjacent to all but at most $\eps(n-1)$ vertices of $G$, containing 
neither a red copy of $P_n$ nor a blue copy of $K_\ell$, then we can 
partition $V(G)$ into $\ell-1$ parts each containing at most $n-1$ 
vertices, such that every edge of $G$ within a part is red, and every 
edge of $G$ between different parts is blue.
\end{lemma}

\begin{proof}
We prove the statement by induction on $\ell$. The case $\ell=2$ is trivial.
Suppose $\ell\geq 3$, and that the statement is true for smaller values of $\ell$.

Let $G$ satisfy the conditions of the lemma, and let $P$ be a maximal red path in $G$, so we have $|P|<n$.
Let $u$ be the first vertex of $P$, and set $X=\Gamma(u)\setminus P$.  By maximality of $P$, every vertex of
$X$ is blue-adjacent to $u$.  It follows that $G[X]$ is a two-coloured graph containing
neither a red copy of $P_n$ nor a blue copy of $K_{\ell-1}$, in which every
vertex is adjacent to all but at most $\eps(n-1)$ vertices of $X$.  Since we have
\[
|X|\geq d(u)-(|P|-1)\geq (\ell-1-\alpha)(n-1)-\eps (n-1)-(n-2)>(\ell-2-\alpha-\eps)(n-1)
\]
and $\eps < (1-\alpha-\eps)/(\ell-1)$, it follows by induction that we can partition $X$ into $\ell-2$ parts,
$X= X_1 \cup \cdots \cup X_{\ell -2}$, such that any edge within a part is red, while any edge between
different parts is blue.  For convenience, we assume that the sets $X_1,\ldots,X_{\ell-2}$ are in increasing
order of size.  Each part contains at most $n-1$ vertices, and therefore the smallest part size $|X_1|$ is at
least $(1-\alpha-\eps)(n-1)$.

Since $1-\alpha-3\eps>0$, we have $\delta(G[X_i])>|X_i|/2$ for each $i$, and
therefore by Dirac's theorem $G[X_i]$ is Hamiltonian for each $i$.  Now observe
that, for each $i$,
$$
|X_i|+|P|\geq |G|-\eps(n-1)-(\ell-3)(n-1)\geq (2-\alpha-\eps)(n-1)>(1+2\eps)(n-1).
$$
It follows that, if $P'$ is a red path in $G[P]$ covering all but at
most $\eps(n-1)$ vertices of $P$, then an endvertex of $P'$ cannot send any red edges to any set $X_i$.

The first vertex $u$ of $P$ has at most $\eps(n-1)$ non-neighbours in $G$
(including itself), so it must be adjacent to at least one vertex $v$ among the
last $\eps(n-1)$ vertices of $P$.  This vertex $v$ is the endvertex of the red
path from $u$ to $v$ following $P$, which covers all but at most $\eps(n-1)$
vertices of $P$. It follows that $v$ has no red neighbours in any set $X_i$.

Since $|X_1|>2\eps(n-1)$, we can find in $X_1$ a common neighbour $x_1$ of the
vertices $u$ and $v$; similarly in $X_2$ we can find a common neighbour of
$u,v$ and $x_1$, and so on until we find $\ell$ vertices forming a clique in
$G$. Since there is no blue $K_\ell$ in $G$, and since all the other edges are
blue, it must be the case that $uv$ is red.  It now follows that every vertex in $P$ is the
endvertex of a path covering all but at most $\eps(n-1)$ vertices of $P$.

For each $1\leq i\leq\ell-2$, let $Y_i$ be the red component of $G$ containing $X_i$, and
let $Y_{\ell-1}$ be the set consisting of the remaining vertices of $G$.  Since $G$
contains no red $P_n$, the $Y_i$ are all distinct, and $P\subset Y_{\ell-1}$.  We
claim that this partition satisfies the desired properties.  By definition, there are
no red edges between any pair of parts.  Since every edge in a part can be
extended to a copy of $K_\ell$ in $G$ by choosing greedily one vertex from each
other part, every edge within a part must be red.  Finally, since each $G[Y_i]$ has minimum
degree at least $|Y_i|-\eps(n-1)>|Y_i|/2$, each component contains a spanning
path by Dirac's theorem: as $G$ contains no $P_n$, this implies that $|Y_i| < n$ for each $i$.
\end{proof}

In the next lemma, we need only the special case of Lemma~\ref{StabLem2} where $G$ is complete --
in this case, the proof above can be streamlined.  We will make use of the full version of the
lemma later.

Our next aim is to prove a similar stability result for $R(P_n^k,H)$.  We do not need the full strength 
of the following result in order to establish the value of $R(P_n^k,H)$, but it is necessary for the proof of Theorem~\ref{NExpAlwaysGood} later.

\begin{lemma} \label{StabPnkH} 

Let $H$ be a graph, $k$ a natural number, and $\eps$ a positive constant at most $1/2|H|^2$. Set 
$n_0 = (1/\eps)(200 k^2 |H| / \eps)^{4k|H|}$. If $n\geq n_0$ and $G$ is a two-coloured complete graph on at 
least $(\chi(H)-1)n- n/6$ vertices, which contains neither a red copy of $P_n^k$ nor a blue copy of $H$, then 
there is a partition $V(G)=V_1\cup \cdots\cup V_{\chi(H)-1}\cup L$ with the following properties.

\begin{itemize}
  \item $|L|\leq\eps n$.
  \item $2n/3 \le |V_i| < n$ for each $i$.  
  \item For each $i$ and $v\in V_i$, $v$ has at most $\eps |V_i|$ blue 
        neighbours in $V_i$.
  \item For each $i$ and $j$, no vertex in $V_i$ has more than 
        $\eps |V_j|$ red neighbours in $V_j$.
\end{itemize}
\end{lemma}
\begin{proof} 
Suppose we are given a graph $H$, a natural number $k$, and a positive 
$\eps$ at most $1/2|H|^2$. We now choose $s = \lceil (128 k^2|H| / \eps)^{4k} \rceil$ and note that $n_0 \ge (3/\eps)(2s)^{|H|}$.



Take any $n\geq n_0$, and let $G$ be a two-coloured complete graph on
$(\chi(H)-1)n - n/6$ vertices which contains neither a red copy of $P^k_n$ nor a blue copy of $H$.

Since $R(K_s,H) \leq R(K_s,K_{|H|}) \leq \binom{s+|H|}{|H|} \le 
(2s)^{|H|} \le \eps n/3$ by the Erd\H{o}s-Szekeres bound~\cite{ErdSzek}, and $G$ contains no blue copy of $H$, it follows 
that any $(\eps n/3)$-vertex set in $G$ contains a red copy of $K_s$. 
Thus we can partition $V(G)$ into disjoint $s$-vertex red cliques $Q_1, 
Q_2,\ldots, Q_M$ and a leftover set $L_1$ with $|L_1|\leq \eps n/3$.

Let $m=\lceil n/s\rceil$. Observe that the number $M$ of cliques is at least
\begin{align*}
 \frac{(\chi(H) -1)n -n/6 - \eps n/3}{s}
   &\ge (\chi(H) -1)(m-1) - \frac{m}{6} - \frac{\eps m}{3}\\
   &\ge \left(\chi(H)-1-\frac{1}{3}\right)(m-1).
\end{align*}

We say that two red cliques $Q_i$ and $Q_j$, $i\neq j$, are
\textit{red-adjacent} if the induced bipartite graph $G[Q_i, Q_j]$
contains a red $K_{2k,2k}$, and \textit{blue-adjacent} otherwise. This gives
us an auxiliary two-coloured complete graph $G^*$ whose nodes are the $M$ red cliques.

Suppose there is a red-adjacent path $Q_{j_1}Q_{j_2} \cdots Q_{j_m}$ on $m$ vertices in $G^*$.  We claim that
the $sm \geq n$ vertices in these $m$ cliques of $G$ can then be covered by a red $k$th power of a path.  Since each
consecutive pair of cliques on the path is red-adjacent in $G^*$, we can find vertex-disjoint red copies of $K_{k,k}$
between each consecutive pair of cliques; now we construct a copy of $P_{sm}^k$ by traversing the sequence
of cliques in order, using the copies of $K_{k,k}$ to step from one clique to the next.  Therefore there is no
red copy of $P_m$ in $G^*$.

If $G^*$ contains a blue-adjacent clique with vertex set 
$\{Q_{j_1}, \dots, Q_{j_{\chi(H)}}\}$, then we can apply 
Lemma~\ref{MyBlowUp}, with $t=\chi(H)$, $d=|H|-1$, $r=2k$, and the 
given value of $s$, to the sets $V_i = Q_{j_i}$.  Each vertex set 
$V'_i$ has size $s - \lfloor 16k^2 s^{1-1/4k}|H| \rfloor\ge s/2\ge |H|$ 
(since $s \ge (64k^2|H|)^{4k}$), and so the auxiliary graph $G'$ 
contains a copy of $H$, and therefore by Lemma~\ref{MyBlowUp} there is
a blue copy of $H$ in $G$.  Therefore there is no blue copy of 
$K_{\chi(H)}$ in $G^*$.

Thus $G^*$ is a two-coloured complete graph on at least $(\chi(H)-1-\frac{1}{3})(m-1)$ vertices,
with neither a red $P_m$ nor blue $K_{\chi(H)}$.  By Lemma~\ref{StabLem2}, applied with $\alpha=1/3$, $G^*$ must
consist of $\chi(H)-1$ red cliques $C^*_1,\ldots,C^*_{\chi(H)-1}$, each with between $2m/3$ and $m-1$ nodes, joined
entirely by blue edges.  These red cliques in $G^*$ correspond to red \textit{clusters} $C_1,\ldots,C_{\chi(H)-1}$ in
$G$, where each cluster contains between $2n/3$ and $n-1$ vertices.

Our plan is to show that we can form the required sets $V_j$ by removing a small number of vertices from each
cluster $C_j$, placing these in the leftover set.  Since we only remove vertices from the clusters, the 
resulting $V_j$ will all have at most $n-1$ vertices.  As long as the leftover set contains at most 
$\eps n \le n/6$ vertices, it follows that each $V_i$ contains at least 
$$
(\chi(H)-1)n - \frac{n}{6} - (\chi(H)-2)(n-1) - \frac{n}{6} > \frac{2n}{3}
$$
vertices.  

Consider a clique $Q$ in the cluster $C_i$. Let $Q'$ be a clique in the 
cluster $C_j$, where $j\neq i$. Then $QQ'$ is a blue-adjacent edge of 
$G^*$: by definition there is no red $K_{2k,2k}$ between $Q$ and $Q'$ 
in $G$. By Theorem~\ref{KSTBound}(a) the number of red edges between $Q$ and $Q'$ is at most $2 s^{2-\frac{1}{2k}} \leq \eps^2 
s^2/6(\chi(H))^2$ -- since $s \ge (12 |H|^2 / \eps^2)^{2k}$.
It follows that there are at most
\[
 \frac{\eps^2 |C_i||C_j|}{6(\chi(H))^2}
\] 
red edges between $C_i$ and $C_j$ in $G$. In particular, at most 
$\eps |C_i|/3(\chi(H))^2$ vertices of $C_i$ can have more than 
$\eps |C_j|/2$ red neighbours in $C_j$.

For each $i$, let $V'_i$ be the set of those vertices of $C_i$ which 
have at most $\eps |C_j|/2$ red neighbours in $C_j$ for each $j\neq i$. 
We have $|C_i|-|V'_i|\leq \eps |C_i|/3\chi(H)$ for each $i$, and so the 
set $L_2 = \displaystyle\bigcup_{i=1}^{\chi(H)-1} (C_i \setminus V'_i)$
of discarded vertices has size at most $\eps n /3$.

Suppose that $V'_1$ contains more than $\eps^2 |V'_1|^2/6\chi(H)$ blue 
edges. This number is at least $2 |V'_1|^{2-1/|H|}$, since 
$|V'_1| \ge |C_1|/2 \ge n/3$, and we comfortably have 
$n \ge 3(12|H|/\eps^2)^{|H|}$. Thus, by Theorem~\ref{KSTBound}(a), 
there is a blue copy $H_1$ of the bipartite graph $K_{|H|,|H|}$ in 
$V'_1$.

We now show that such a blue bipartite graph inside $V'_1$ can be extended to a blue copy
of the $\chi(H)$-partite graph $K_{|H|,\ldots,|H|}$ in $G$, by taking a suitable set of
$|H|$ vertices from each other $V'_j$.

The number of vertices in $V'_2$ which send red edges to any vertex of $H_1$ is at most
$2|H|\eps |C_2|/2\leq |V'_2|-|H|$; in particular, there are $|H|$ vertices of $V'_2$ which
each send blue edges to every vertex of $H_1$. Thus we have a blue copy $H_2$ of the
tripartite graph $K_{|H|,|H|,|H|}$ in $V'_1 \cup V'_2$.

Repeating this argument for each $V'_3,\ldots,V'_r$ successively, using that, at each stage,
$(\chi(H)-1)|H|\eps |C_j|/2\leq |V'_j|-|H|$ -- this follows because $|V'_j| \ge |C_j|/2$ and
$\eps |H|^2 \le 1/2$ -- we find eventually a blue copy
$H_{\chi(H)-1}$ of the $\chi(H)$-partite graph $K_{|H|,\ldots,|H|}$, as claimed.
This graph contains $H$, which is a contradiction.

It follows that $V'_1$ contains at most $\eps^2 |V'_1|^2/6\chi(H)$ blue 
edges, and thus we can delete a set of at most $2\eps |V'_1|/3\chi(H)$ 
vertices of $V'_1$ to obtain a set $V_1$ such that, for each 
$v \in V_1$, $v$~has at most $\eps |V'_1|/2$ blue neighbours in $V_1$.


By symmetry, for each $2\leq i\leq \chi(H)-1$, one may remove 
$2\eps |V'_i|/3\chi(H)$ vertices from $V'_i$ to obtain a set $V_i$ such 
that each $v\in V_i$ has at most $\eps |V'_i|/2$ neighbours in $V_i$.
The set $L_3=\displaystyle\bigcup_{i=1}^{\chi(H)-1}(V_i\setminus V'_i)$ 
of vertices discarded in this step is again of size at most $\eps n /3$.

Now set $L=L_1 \cup L_2 \cup L_3$, so $|L| \le \eps n$.  Note also that 
each set $V_i$ has size at least $|C_i|/2$, so every vertex in each 
$V_i$ has at most $\eps |V_i|$ blue neighbours in $V_i$, and at most
$\eps |V_j|$ red neighbours in each other $V_j$.

Therefore the partition $V(G)=V_1\cup\cdots\cup V_{\chi(H)-1}\cup L$ is 
as desired.
\end{proof}

Given a graph $G$ possessing a partition as in Lemma~\ref{StabPnkH}, 
one can easily find (by the Sauer-Spencer Theorem~\cite{SauSpe}) in 
$V_i$ a red copy of any graph on $|V_i|$ vertices with maximum degree 
at most $1/2\epsilon$. However we would like to find a red copy of 
$P_n^k$, and our method of proof only gives sets $V_i$ of size 
$(1-\eps)n$.  So, in order to establish the exact value of 
$R(P_n^k, H)$, we will have to find a way either to incorporate the 
vertices of the leftover set $L$ into the sets $V_i$, or to show that, 
when this is not possible, $G$ contains a blue copy of $H$. To do this, 
we give an embedding lemma based on the Sauer-Spencer Theorem; we shall 
apply it in the case where $F$ is the red graph with vertex set 
consisting of one of the sets $V_i$ together with some vertices of $L$, 
and $J = P_n^k$.

\begin{lemma}\label{GenEmbed} 
 Given a natural number $\Delta \ge 1$ and any $0<\eps<1/(\Delta^2+4)$, 
 let $F$ be an $n$-vertex graph in which every vertex has degree at 
 least $3\Delta\eps n$, and all but at most $\eps n$ vertices have 
 degree at least $(1-2\eps)n$. Let $J$ be any $n$-vertex graph with 
 $\Delta(J)\leq\Delta$. Then $J\subset F$.
\end{lemma}

\begin{proof}
 Let $P\subset V(F)$ be those vertices of $F$ with degree less than 
 $(1-2\eps)n$. Since $|P| \leq \eps n < n/(\Delta^2+1)$, we can find 
 a set $I\subset V(J)$ with $|I|=|P|$, and such that no two vertices of 
 $I$ are either adjacent or have any common neighbour in $J$ (we simply 
 choose vertices satisfying the conditions greedily).

 Let $\phi:I\rightarrow P$ be any bijection from $I$ to $P$. We 
 construct now a partial embedding $\phi'$ of $I$ together with all its 
 neighbours $\Gamma(I)$ into $F$, extending $\phi$. We do this by taking
 an enumeration $\{x_1,\ldots,x_m\}$ of $\Gamma(I)$ and, for each $i$ 
 in turn, choosing a vertex $y_i\in V(F)$ to be $\phi'(x_i)$ with the 
 following properties.

 First, we require that $y_i\notin \phi'(I\cup\{x_1,\ldots,x_{i-1}\})$.  At most $|I\cup\Gamma(I)| \leq (\Delta+1)\eps n$ vertices of $F$ fail 
 to satisfy this condition.

 Second, for any vertex $v$ of $I\cup\{x_1,\ldots,x_{i-1}\}$ which is 
 adjacent to $x_i$, $y_i$ must be adjacent to $\phi'(v)$. Observe that, 
 in $J$, there is exactly one vertex of $I$ adjacent to $x_i$ and at 
 most $\Delta-1$ other vertices (not in $I$) adjacent to $x_i$. It 
 follows that at most $(n- 3\Delta\eps n) + (\Delta-1) 2\eps n$ 
 vertices of $F$ fail to satisfy this condition.

 Since $n > (\Delta+1)\eps n + (n-3\Delta\eps n) + (\Delta-1)2\eps n$, 
 we will never become stuck, and the desired extension $\phi'$ exists.

 Now let $\psi$ be a bijection from $V(J)$ to $V(F)$ extending $\phi'$ 
 and such that $|\{ e \in E(J) : \psi(e) \notin E(F)\}|$ is
 minimised. We claim that the number of such `bad' edges is in fact 
 zero; that is, $\psi$ is an embedding of $J$ into $F$, as desired.

 Suppose this were false: then there is an edge $ab$ of $J$ such that
 $\psi(ab)\notin E(F)$. Because $\psi$ extends $\phi'$, and $\phi'$ is 
 an embedding, at least one of $a$ and $b$, say $b$, is not in 
 $I\cup\Gamma(I)$.

 Because $b\notin I\cup\Gamma(I)$, every neighbour $v$ of $b$ in $J$ 
 satisfies $d(\psi(v)) \geq (1-2\eps) n$. In particular, there are at 
 least $(1 - 2\Delta\eps) n$ vertices of $F$ which are adjacent to 
 $\psi(v)$ for every neighbour $v$ of $b$.

 The number of vertices of $J$ which have a neighbour in
 $\psi^{-1}\big(V(F)\setminus \Gamma(\psi(b))\big)$ is at most  
 $2\Delta\eps n$, since $|V(F)\setminus\Gamma(\psi(b))| \leq 2\eps n$ 
 and $J$ has maximum degree $\Delta$.

 Since $(1-2\Delta\eps)n - 2\Delta\eps n>0$, there is a vertex $c$ of 
 $J$ such that $\psi(c)$ is adjacent to $\psi(v)$ for each neighbour 
 $v$ of $b$ and such that $c$ has no neighbours in 
 $\psi^{-1}(V(F)\setminus \Gamma(\psi(b)))$; that is, for each
 neighbour $v$ of $c$, $\psi(v)$ is a neighbour of $\psi(b)$.

 Now let $\psi':V(J)\rightarrow V(F)$ be defined as follows. 
 \[
  \psi'(x)=
  \begin{cases}
   \psi(c) & \mbox{ if $x=b$,}\\
   \psi(b) & \mbox{ if $x=c$,}\\
   \psi(x) & \mbox{ if $x\neq b, c$.}
  \end{cases}
 \]
 In other words, we swap the targets under $\psi$ of $b$ and~$c$.

 By construction, any edge $e$ of $J$ which meets $b$ or $c$ is mapped 
 to an edge of $F$ by $\psi'$. Since $\psi'(x)=\psi(x)$ when 
 $x\neq b,c$, any bad edge of $\psi'$ not meeting $b$ or $c$ is also a 
 bad edge of~$\psi$; thus $\psi'$ has at least one fewer bad edge 
 (namely $ab$) than $\psi$, which contradicts minimality of $\psi$.
\end{proof}

Now we can give the proof of Theorem~\ref{2alwaysgood}.

\begin{proof}[Proof of Theorem~\ref{2alwaysgood}]

Given a graph $H$ and a natural number $k\ge 2$, set $\eps= 1/8k^2 |H|$ 
and $n_0 = (20k|H|)^{16k|H|} \ge (1/\eps)(200 k^2 |H|/\eps)^{4k|H|}$. 
Take any $n \ge n_0$, and let $G$ be a two-coloured complete graph on 
$(\chi(H)-1)(n-1)+\sigma(H)$ vertices.

By Lemma~\ref{StabPnkH}, if $G$ contains neither a red copy of $P_n^k$ 
nor a blue copy of $H$, then we have a partition 
$V(G) = V_1\cup\cdots\cup V_{\chi(H)-1}\cup L$ such that the following 
are true.
\begin{itemize}
  \item $|L|\leq\eps n$.
  \item $2n/3 \le |V_i| < n$ for each $i$.
  \item For each $i$ and $v\in V_i$, $v$ has at most $\eps |V_i|$ blue 
        neighbours in $V_i$.
  \item For each $i$ and $j$, no vertex in $V_i$ has more than 
        $\eps |V_j|$ red neighbours in $V_j$.
\end{itemize}

For each $i$, let $C_i$ be the set of vertices of $L$ which send at least 
$6k\eps n$ red edges to~$V_i$.

Suppose that for some $i$ we have $|V_i\cup C_i|\geq n$. An application 
of Lemma~\ref{GenEmbed} (with $F$ the red graph induced on 
$V_i \cup C_i$, $J=P_n^k$ and $\Delta = 2k$, noting that indeed 
$\eps < 1/(4k^2+4)$) shows that there is a red copy of $P_n^k$ in $G$, 
which is a contradiction. It follows that, for each~$i$, we have 
$|V_i\cup C_i|\leq n-1$; so
\[
 \bigg|V(G)\setminus\bigcup_{i=1}^{\chi(H)-1}(V_i\cup C_i)\bigg|
  \geq
 |V(G)|-(\chi(H)-1)(n-1)=\sigma(H)~.
\]

Thus there is a set $S$ of $\sigma(H)$ vertices in $L$, each of which 
sends at least $|V_i|-6k\eps n$ blue edges to $V_i$ for each $i$. Take 
a $\chi(H)$-colouring $c$ of $H$ in which the part with colour 
$\chi(H)$ has $\sigma(H)$ vertices. We construct a blue copy of $H$ in 
$G$ greedily as follows. Let $T_1$ be a set of $|c^{-1}(1)|$ vertices 
in $V_1$ each of which is blue-adjacent to every member of $S$; let 
$T_2$ be a set of $|c^{-1}(2)|$ vertices of $V_2$ each of which is 
blue-adjacent to every member of $S\cup T_1$, and so on. The number of 
vertices of $V_i$ which are red-adjacent to some member of $S$ is at 
most $\sigma(H)6k\eps n$; the number which are red-adjacent to some 
previously chosen vertex of $T_1\cup\cdots\cup T_{i-1}$ is at most 
$|H|\eps |V_i|$.

Since $\sigma(H) 6k\eps n + |H| \eps|V_i| \leq (6k+1)|H|\eps n < n/2 < 
|V_i|-|H|$, we never become stuck. We obtain a blue complete $\chi(H)$-partite graph which contains $H$. This completes the proof.
\end{proof}

The proof above is not the simplest way to obtain Theorem~\ref{2alwaysgood}: it is possible to work directly with the structure of the auxiliary graph $G^*$ constructed in the proof of Lemma~\ref{StabPnkH}. However, this proof lends itself to the generalisation required to prove Theorem~\ref{NExpAlwaysGood}.

Our upper bound on $n_0$ can be improved in two specific cases.  If we
wish to find $R(P_n,H)$ or $R(C_n,H)$, then we can change the definition of `red-adjacent' from requiring a red complete bipartite graph to requiring only two or three (respectively) disjoint red edges between the cliques. This then allows us to avoid the use of Lemma~\ref{MyBlowUp} and obtain results with cliques on $|H|$ or $2|H|$ vertices, and so (with a little care) for $n_0 = 2^{2|H|}$ or $n_0 = 2^{3|H|}$ respectively.  Our method cannot even get started for $n \le 2^{|H|/2}$, since Erd\H{o}s~\cite{ErdLow} has proved that the Ramsey number $R(K_\ell,K_\ell)$ is bigger than $2^{\ell/2}$.

A consequence of a lower bound of Spencer~\cite{SpLB} on $R(K_\ell, K_m)$ is that, for $k\ge 2$, there exists $c_k>0$ such that the formula $R(P_n^k,H)=(\chi(H)-1)(n-1)+\sigma(H)$ fails when
\[
 n < c_k\left(\frac{|H|}{\ln |H|}\right)^\frac{k^2+k-2}{2k-2}\,,
\]
since $K_{k+1}\subset P_n^k$; in general we really do require $n$ to be 
much bigger than $|H|$.

However, we conjecture that for $P_n$ and $C_n$ the formula holds for 
quite small values of $n$: when $n\geq |H|$.  We note that one case of 
this conjecture -- that $R(C_n,K_\ell)=(\ell-1)(n-1)+1$ holds for
$n\geq \ell$ -- is an old conjecture of Erd\H{o}s, Faudree, Rousseau 
and Schelp~\cite{EFRS}.  Even in this simple case the best known result 
is that the formula holds for $n\geq 4\ell+2$, due to 
Nikiforov~\cite{Nikiforov}.

\section{Poor expanders are always-good} \label{NExpFix}

 In this section we prove Theorem~\ref{NExpAlwaysGood}. We use the same 
 general approach as in the previous section, but here we make use of 
 the Szemer\'edi Regularity Lemma~\cite{SzReg}, together with a theorem 
 of B\"ottcher, Schacht and Taraz~\cite{BST}, in order to replace the 
 graph $P^k_n$ in Theorem~\ref{2alwaysgood} with a general graph $G$ of 
 bounded maximum degree and fairly small bandwidth.

 Given $\eps>0$, let $G$ be an $n$-vertex graph. For $U$ and $V$ 
 disjoint subsets of $V(G)$, let $e(U,V)$ denote the number of edges of 
 $G$ between $U$ and $V$, and define the \textit{density} $d(U,V)$ of 
 the pair $(U, V)$ as
 \[
   d(U,V)=\frac{e(U,V)}{|U||V|}~.
 \]
 We call $(U,V)$ an \textit{$\eps$-regular pair} if, for all pairs of 
 subsets $U'\subset U$ and $V'\subset V$ with $|U'|\geq\eps|U|$ and 
 $|V'|\geq\eps|V|$, we have $|d(U',V')-d(U,V)|<\eps$.

 Suppose we have a partition $V(G)=Z_0\cup Z_1\cup\cdots\cup Z_r$ 
 satisfying the following properties.
 \begin{itemize}
  \item $|Z_0|\leq\eps n$.
  \item For each $1\leq i\leq r$, there are at most $\eps r$ sets $Z_j$ 
  			such that $(Z_i,Z_j)$ is not $\eps$-regular.
  \item $|Z_1|=|Z_2|=\cdots=|Z_r|$.
 \end{itemize}
 Then we call this partition \textit{$\eps$-regular}. We call the 
 partition classes \textit{clusters} and we refer to $Z_0$ as the 
 \textit{exceptional cluster}. In his seminal work~\cite{SzReg},
 Szemer\'edi proved that every sufficiently large graph has an 
 $\eps$-regular partition in which the number of clusters is bounded by 
 a function of $\eps$ and is independent of the number of vertices. We 
 shall use this result in the following form.

 \begin{theorem}[Regularity Lemma]\label{SzLem}
  For any $\eps>0$ and $k_0$, there exist $K$ and $n_0$ such that, 
  whenever $n\geq n_0$ and $G$ is an $n$-vertex graph, $G$ possesses an 
  $\eps$-regular partition with between $k_0$ and $K$ clusters.
 \end{theorem}

 When we have an $\eps$-regular partition of a graph $G$, we associate 
 with it a \textit{cluster graph} $R(G)$ whose nodes are the clusters 
 of the partition (excluding $Z_0$) and whose edges correspond to 
 $\eps$-regular pairs of clusters -- possibly only those whose density is 
 above some given \textit{density threshold} $d$. One can easily prove that under 
 very simple conditions, if $R(G)$ contains a fixed graph $H$, then $G$ 
 must also contain $H$ as a subgraph. This is summarized in the 
 following lemma (see, for instance, Diestel~\cite{Diestel}).
 \begin{lemma}\label{fact:embed}
  For every $d>0$, $\Delta\geq 1$, there exists
  $\epsilon_{\subref{fact:embed}} = \epsilon_{\subref{fact:embed}}(d, \Delta) 
  \leq 1/2$ with the following property. 
  Let $G$ be an $n$-vertex graph and $R(G)$ be a cluster graph with 
  $r\leq d^\Delta n/4$ clusters, $\eps\leq\eps_{\subref{fact:embed}}$, 
  and with density threshold $d$. 
  Then, for every graph $H$ with 
  $\Delta(H)\leq\Delta$, if $R(G)$ contains $H$ as a subgraph, then $G$ 
  also contains $H$.
 \end{lemma}
 To handle bounded degree graphs with small bandwidth, we require the 
 following theorem, essentially due to B\"ottcher, Schacht and
 Taraz~\cite{BST}.

 \begin{theorem} [B\"ottcher, Schacht and Taraz~\cite{BST}]  
 \label{BSTThm}
%

 For any $\mu, \gamma > 0$ and for any natural numbers $\chi$ and $\Delta$ 
 there exists $\eps_{\subref{BSTThm}}>0$ such that, for all 
 $0<\eps\leq  \eps_{\subref{BSTThm}}$, there is a $K_{\subref{BSTThm}}$ such 
 that, for all $K\geq K_{\subref{BSTThm}}$, there exist $\beta>0$ and 
 $n_{\subref{BSTThm}}$ such that the following holds for all 
 $n\geq n_{\subref{BSTThm}}$ and $0\leq\eta< 1$.

 Let $F$ be an $n$-vertex graph, and $R(F)$ the cluster graph 
 corresponding to an $\eps$-regular partition of $F$ with $r\leq K$ 
 parts whose edges correspond to $\eps$-regular pairs of density at 
 least $\gamma$. Suppose that in $R(F)$ there is a copy of 
 $P^{\chi-1}_{(\eta+\mu) r}$ with the further property that every 
 $\chi$-clique is contained in a $(\chi+1)$-clique of $R(F)$. Then 
 whenever $G$ is an $\eta n$-vertex graph with maximum degree $\Delta$, 
 chromatic number $\chi$ and bandwidth $\beta n$, we have $G\subset F$.
\end{theorem}

To be specific, Lemma 8 of~\cite{BST} provides a graph homomorphism with
certain additional desirable properties from $G$ to $R(F)$, in particular that
no vertex of $R(F)$ is the target of `too many' vertices of $G$. Since in our
situation we seek to embed only $\eta n$ vertices into $(\eta+\mu)r$ clusters
(and therefore we have at least $(\eta+\mu/2)n$ vertices in the union of the
clusters of the $P^{\chi-1}_{(\eta+\mu) r}$ in $R(F)$ but only $\eta n$ in
$G$, as opposed to the requirement in~\cite{BST} for a spanning embedding), we
may in particular presume that for each $i$ the homomorphism allocates at most
$(1-\mu/4)|Z_i|$ vertices of $G$ to the cluster $Z_i$. One can then complete
the embedding of $G$ into $F$ by using the Blow-up Lemma of Koml\'os,
S\'ark\"ozy and Szemer\'edi~\cite{Blowup} (again following the method
of~\cite{BST}, Proof of Theorem 2). It should be emphasized that the major
source of difficulty in~\cite{BST} is the requirement for a spanning
embedding: obtaining an embedding covering a $(1-\mu)$-fraction of $F$
(essentially our situation) is relatively trivial.

 Whenever we use this, we will in fact find a copy of 
 $P^{\chi}_{(\eta+\mu)r}$ in $R(F)$; this certainly contains a copy of 
 $P^{\chi-1}_{(\eta+\mu)r}$ in which every $\chi$-clique extends to a 
 $(\chi+1)$-clique. For convenience, we presume the parameter 
 $n_{\subref{BSTThm}}$ is chosen to be at least as large as required for 
 Theorem~\ref{SzLem} to provide an $\eps$-regular partition.

 In our proof of Theorem~\ref{NExpAlwaysGood}, we shall use Theorem 
 \ref{BSTThm} with $\eta=1/(\chi(H)-1)$. Since it is required that 
 $\eta < 1$, we need to treat separately the case when $H$ is bipartite.
It turns out that, in this case, we need no restriction on the bandwidth
of $G$. 

For the reader's convenience, we restate Theorem~\ref{bipartite}.  


\begin{reptheorem}{bipartite}
 For every bipartite graph $H$ and natural number $\Delta$, there exists 
 $n_0$ such that, for all $n\geq n_0$, every connected $n$-vertex graph 
 $G$ with $\Delta(G)\leq \Delta$ satisfies 
 $R(G,H)=(\chi(H)-1)(n-1)+\sigma(H) = n + \sigma(H) -1$.
\end{reptheorem}

\begin{proof}
 Fix a bipartite graph $H$ and a positive integer $\Delta$. Set 
 $\eps= 1/(8\Delta^2\sigma(H))$ and $n_0=|H|(16/\eps^2)^{|H|}$. For 
 $n>n_0$, let $G$ be a connected $n$-vertex graph with maximum degree 
 at most $\Delta$, set $N=n-1+\sigma(H)$, and let $F$ be a $2$-coloured 
 complete graph on $N$ vertices.

 If the blue subgraph of $F$ does not contain $K_{|H|, |H|}\supset H$, 
 by Theorem \ref{KSTBound}(a), $F$ contains at most 
 $2N^{2-\frac{1}{|H|}}$ blue edges. Hence, the set $X$ of vertices with blue degree at
 least $2\eps n$ has $|X| \le \eps n$.

 Let $Y\subset X$ be the vertices of $X$ which have blue degree greater 
 than $N-n/(2\sigma(H))$. Any $\sigma(H)$ vertices of $Y$ have a common 
 blue neighbourhood containing at least 
 $N-\sigma(H) n/(2\sigma(H)) \geq n/2>|H|$ vertices, and therefore, if
 $|Y|\geq\sigma(H)$, there exists a blue copy of $H$. (the graph $H$ has a
 bipartition with one partite set of size $\sigma(H)$ and the other of size 
 $|H|-\sigma(H)$.) Thus, $|Y|\leq\sigma(H)-1$. 

 We apply Lemma~\ref{GenEmbed} with $\Delta$, $\eps=1/(8\Delta^2\sigma(H)) < 
 1/(\Delta^2+4)$, and $J=G$ to the red graph induced on 
 $V(F)\setminus Y$. All the assumptions hold: the set $V(F)\setminus Y$ has 
 at least $N-|Y|\geq n$ vertices, every vertex has red degree at least 
 $n/(2\sigma(H))\geq 3\Delta\eps n$, and all but at most $\eps n$ 
 vertices (those in $V(F)\setminus X$) have red degree at least 
 $n-2\eps n$. Hence, $F$ contains a red copy of $G$.
\end{proof}

 Our general strategy for proving Theorem~\ref{NExpAlwaysGood} is very 
 similar to that in the proof of  Theorem~\ref{2alwaysgood}. We will need to be 
 able to apply a version of our stability result, Lemma~\ref{StabPnkH}, to cluster graphs. 
 This means that we need the following variant of Lemma~\ref{StabPnkH}, 
 where our two-coloured graphs are not complete, but rather have minimum 
 degree $(1-\eps)n$ for some $\eps>0$ (whose size we may choose as 
 small as we desire). 

\begin{lemma}[Modified Lemma \ref{StabPnkH}] \label{StabPnkH:modified} 
 Let $H$ be a graph, $k$ a natural number, and $\eps'$ a positive
 constant at most $1/2|H|^2$. There exists $\eps_{\subref{StabPnkH:modified}}<\eps'$
 such that for every positive $\eps< \eps_{\subref{StabPnkH:modified}}$ there is 
 an $n_{\subref{StabPnkH:modified}}$ for which the following holds. 
 If $n\geq n_{\subref{StabPnkH:modified}}$ 
 and $G$ is a two-coloured graph on $N\geq (\chi(H)-1)n- n/6$ vertices with 
 $\Delta(\bar{G})<\eps n$, which contains neither a red copy of $P_n^k$
 nor a blue copy of $H$, then there is a partition 
 $V(G)=V_1\cup \cdots\cup V_{\chi(H)-1}\cup L$ with the following properties.

\begin{itemize}
  \item $|L|\leq\eps' n$.
  \item For each $i$, $2n/3\leq |V_i|<n$.
  \item For each $i$ and $v\in V_i$, $v$ has at most $\eps' |V_i|$ blue 
        neighbours in $V_i$.
  \item For each $i$ and $j$, no vertex in $V_i$ has more than 
        $\eps' |V_j|$ red neighbours in $V_j$.
\end{itemize}
\end{lemma}

\begin{proof} (Sketch)
This is an entirely straightforward modification of the proof of 
Lemma~\ref{StabPnkH}. There are two changes which must be made.

First, we can
no longer use the Erd\H{o}s-Szekeres bound $R(K_s,K_{|H|})\leq \binom{|H|+s}{|H|}$ to find red $s$-cliques. It is easy to prove (albeit with slightly worsened bounds) that, if $n$ is 
large enough, then any two-coloured graph with minimum degree 
$(1-\eps)n$ contains either $K_s$ or $K_{|H|}$.


Second, the
auxiliary graph $G^*$ must be defined slightly differently. Just as before, we
say two cliques $Q_i$ and $Q_j$ are red-adjacent if the induced bipartite graph
$G[Q_i,Q_j]$ contains a red $K_{2k,2k}$. If however $Q_i$ and $Q_j$ are not
red-adjacent, then we have two possibilities. If there are at least
$2\sqrt{\eps}|Q_i||Q_j|$ non-edges of $G$ in $G[Q_i,Q_j]$ then $Q_i$ and $Q_j$
are non-adjacent in $G^*$. If $Q_i$ and $Q_j$ are not red-adjacent, and the
number of non-edges of $G$ in $G[Q_i,Q_j]$ is less than
$2\sqrt{\eps}|Q_i||Q_j|$, then $Q_i$ and $Q_j$ are blue-adjacent. Observe that
if there is a vertex $Q_i$ with $\sqrt{\eps}v(G^*)$ non-neighbours in $G^*$,
then there are at least $2\sqrt{\eps}|Q_i|(v(G^*)-1)|Q_i|$ nonedges in $G^*$
adjacent to $Q_i$, so a vertex $v\in Q_i$ of minimum degree in $G$ has more
than $\eps n$ non-neighbours in $G$. This contradiction yields
$\Delta(\bar{G^*})<\sqrt{\eps}v(G^*)$.

 The rest of the proof goes through unchanged: note that since $G^*$ is now not
 a complete graph, we use the full strength of Lemma~\ref{StabLem2} in this
 setting.
\end{proof}
 
We are now ready to prove the main result of this section.  
 
\begin{proof}[Proof of Theorem~\ref{NExpAlwaysGood}]
 Fix a graph $H$ with $\chi(H)\geq 3$ and an upper bound $\Delta$ on 
 the maximum degree of~$G$.  We need to show that, if $\beta>0$ is 
 sufficiently small, then for all sufficiently large $n$, every 
 connected $n$-vertex graph $G$ with $\Delta(G)\leq \Delta$ and 
 $\bw(G)\leq\beta n$ satisfies $R(G,H)=(\chi(H)-1)(n-1)+\sigma(H)$.

 We set $\gamma=\frac{1}{200\Delta^2|H|}$ and 
 $\mu=\frac{1}{12\chi(H)^2}$, and choose $\eps'$ such that 
 $$
  \eps'< \min\left\{\epsilon_{\subref{fact:embed}}(1/2, \Delta),
                   \frac{1}{2|H|^2},  
                   \frac{1}{\Delta^2+4}
            \right\}
      \ \text{ and } \ 
  4(\chi(H)+3)\eps'+8\gamma
  <
  \min\left\{\frac{1}{\Delta^2+4},\frac{1}{(6\Delta+2)|H|}\right\}
 $$ 
 hold. We then choose  $\eps$ such that $\eps<\eps'$, 
 $\eps<\eps_{\subref{BSTThm}}(\mu, \gamma, \chi(H), \Delta)$, and 
 $\eps<\eps_{\subref{StabPnkH:modified}}(H,\chi(H),\eps')$. 
 
 Let $k_0$ be such that $k_0\geq 2/\eps$, 
 $k_0> K_{\subref{BSTThm}}(\mu,\gamma,\chi(H),\Delta,\eps)$,
 and $k_0 > n_{\subref{StabPnkH:modified}}(H,\chi(H),\eps',\eps)$.  Let $K$ and $n_0$ be  
 constants such that the conclusion of Theorem~\ref{SzLem} holds, 
 with parameters $\eps$ and $k_0$, so in particular 
 $K\geq k_0> K_{\subref{BSTThm}}(\mu,\gamma,\chi(H),\Delta,\eps)$.  Now let
 $\beta>0$ and $n_{\subref{BSTThm}}\ge n_0$ be constants such that the conclusion of
 Theorem~\ref{BSTThm} holds.  

 Take any $n\geq \max( n_{\subref{BSTThm}}, K 2^{\Delta+2})$, and set $\eta=1/(\chi(H)-1)\leq 1/2$. 
 Let $G$ be any connected $n$-vertex graph with $\bw(G)\leq\beta n$
 and $\Delta(G)\leq \Delta$, so $\chi(G)\leq \Delta+1$.

 Let $F$ be a two-coloured complete graph on 
 $(\chi(H)-1)(n-1)+\sigma(H)$ vertices. Our aim is to prove that $F$ 
 contains either a red copy of $G$ or a blue copy of $H$.

 By applying Theorem~\ref{SzLem} to the red graph of $F$, we obtain an
 $\eps$-regular partition, and hence a cluster graph $R(F)$ with some 
 number $r$ of vertices, $k_0 \le r \le K$. By moving the vertices of at most $\chi(H)$ 
 clusters to the exceptional set $Z_0$, we may assume that 
 $r=(\chi(H)-1)m$ for some integer $m$. When $(A,B)$ is a pair of 
 clusters which is $\eps$-regular, we have an edge $AB$ in $R(F)$.  
 This graph $R(F)$ is very nearly complete: each vertex has degree 
 at least  $(1-\eps)r$.
 When the density of red edges in $(A,B)$ is more than $\gamma$, we 
 colour $AB$ red, otherwise we colour it blue. Notice that if $AB$ is 
 blue, then the density of blue edges in $(A, B)$ is at least~$1/2$.

 If there is a blue copy of $H$ in $R(F)$, then, by Lemma~\ref{fact:embed} with
 $d=1/2$, $F$ contains a blue copy of $H$, and we are done.
 
 Also, if there is a red copy of $P^{\chi(H)}_{(\eta+\mu)r}$ in $R(F)$, 
 then, by Theorem~\ref{BSTThm}, there is a red copy of $G$ in $F$, and 
 again we are done.

 Observe that $(\chi(H)-1)(\eta+\mu)r-(\eta+\mu)r/6 < r$. Thus if
 $R(F)$ contains neither a red $P^{\chi(H)}_{(\eta+\mu)r}$ nor a blue 
 $H$, then, by Lemma~\ref{StabPnkH}, we have a partition 
 $V(R(F))=W_1\cup\cdots\cup W_{\chi(H)-1}\cup L'$ with the following 
 properties.
\begin{itemize}
  \item $|L'|\leq\eps' (\eta+\mu)r$.
  \item For each $i$, $2(\eta+\mu)r/3 \le |W_i| < (\eta+\mu)r$. 
  \item For each $i$ and $v\in W_i$, $v$ has at most $\eps' |W_i|$ blue
  neighbours in $W_i$.
  \item For each $i$ and $j$, no vertex in $W_i$ has more than $\eps' |W_j|$ red
  neighbours in $W_j$.
\end{itemize}

 Consider the cluster $A\in W_i$, and let $j\neq i$.  Set 
 $$
 \delta
  = 
 2\,(4\eps'+2\eps(\chi(H)-1)+4(\gamma+\eps)),
 $$
 and let $A'$ be the set of vertices in $A$ which send more than  
 $\delta n/2$ red edges to $W_j$. Suppose that $|A'|\geq \eps|A|$.
 
 In $R(F)$, $A$  sends red edges to at most $\eps' |W_j|$ clusters of $W_j$. To these clusters, $A'$ sends at most $\eps' |W_j|\cdot |A'|(N/r)$ red edges.
 
 There are further at most $\eps r$ clusters of $W_j$ which are not 
 adjacent in $R(F)$ to $A$, corresponding to non-$\eps$-regular pairs 
 in $F$. To these clusters, $A'$ sends at most $\eps r\cdot |A'|(N/r)$  
 red edges.
 
 The remaining clusters of $W_j$ are linked by blue edges in $R(F)$ to 
 $A$. Hence, their red density is at most $\gamma$ and they are 
 $\eps$-regular. Using $\eps$-regularity, the total number of red edges 
 from $A'$ to these clusters is bounded by 
 $|W_j|\cdot (\gamma+\eps)|A'|(N/r)$.

 Using that $|W_i|< (\eta+\mu)r=m+\mu r< 2m$ and 
 $N/r = n/m + \sigma(H)/r< 2n/m$, we obtain
 $$
  e_{\rm red}(A',\bigcup W_j)
   <
  (4\eps'+2\eps(\chi(H)-1)+4(\gamma+\eps))|A'|n=\delta |A'|n/2.
 $$
 On the other hand, it follows from the definition of $A'$ that
 $e_{\rm red}(A,\bigcup W_j)>\delta n |A'|/2$, which is a 
 contradiction. Hence, $|A'|<\eps |A|$.



 Since this holds for each cluster of $W_i$ and every $j\neq i$, we can 
 remove at most $(\chi(H)-2)\eps|\bigcup W_i|$ vertices from 
 $\bigcup W_i$ to obtain a set $V_i$ of vertices of $F$ which sends at 
 most $\delta n/2$ red edges to $\bigcup W_j$ for every $j\neq i$.

 Since $|V_i|>n/2$ for each $i$, we certainly have that for each 
 $i\neq j$, every vertex in $V_i$ has at most $\delta|V_j|$ red 
 neighbours in $V_j$.

 Given $i$, by an identical argument to that in the proof of
 Lemma~\ref{StabPnkH}, if there are more than $\delta^2|V_i|^2/6$ blue 
 edges in $V_i$, then $V_i$ contains a blue copy of $K_{|H|,|H|}$ which 
 we can extend to a blue copy of $H$ in $F$. It follows that we can 
 remove at most $2\delta|V_i|/3$ vertices from $V_i$ to obtain a set 
 $V'_i$ such that every vertex in $V'_i$ has at most $\delta|V_i|/2$ 
 blue neighbours in $V_i$. Thus, for each $i$, every vertex in $V'_i$ 
 has at least $(1-\delta)|V'_i|$ red neighbours in $V'_i$.

 We can now complete the proof in an identical fashion to the proof of
 Theorem~\ref{2alwaysgood}. We let the set $L$ contain all those 
 vertices of $F$ which are in no set $V'_i$. We let $C_i$ be the set of 
 vertices in $L$ which send at least $3\Delta\delta n$ edges to $V'_i$. 
 By Lemma~\ref{GenEmbed}, if for some $i$ we have 
 $|V'_i\cup C_i|\geq n$, then we can find a red copy of $G$ in $F$.

 But if for each $i$ we have $|V'_i\cup C_i|\leq n-1$, then we can
 find a set $S$ of $\sigma(H)$ vertices of $L$ which each sends at least
 $(1-3\Delta\delta)n$ blue edges to each set $V'_i$. As in the proof of
 Theorem~\ref{2alwaysgood}, we can now construct a blue copy of $H$ in 
 $F$ greedily. This completes the proof.
\end{proof}

\section{Multi-colour problems}\label{Mcol}

 In this section, we prove Theorem~\ref{MAGisAG}, stating that if  
 $\mathcal{G}$ is an always-good family of graphs, then $\mathcal{G}$ 
 is also multicolour-always-good; that is, given $r\geq 2$ and graphs
 $H_1,\ldots,H_r$, there are integers $W$ and $Z$ (not depending on 
 $\mathcal{G}$) such that, for all sufficiently large 
 $G\in\mathcal{G}$, $R(G,H_1,H_2,\ldots,H_r)=W(|G|-1)+Z$.

 Our first step is to give explicit definitions of the constants $W$ 
 and $Z$ as the solutions to two further Ramsey-type problems involving 
 $H_1, \dots, H_r$.

 The first is a variant of the standard Ramsey problem.  We define the 
 \textit{homomorphism Ramsey number} $R_{\rm hom}(H_1,\ldots,H_r)$ to 
 be the smallest $N$ such that, if $F$ is any $r$-coloured complete 
 graph on $N$ vertices, there exists a colour $i$ such that there is a graph 
 homomorphism from $H_i$ into the $i$th colour subgraph of $F$. Then we 
 let $$W=R_{\rm hom}(H_1,\ldots,H_r)-1.$$  It is clear that
 $R_{\rm hom}(H_1,\ldots,H_r)\leq R(H_1,\ldots,H_r)$, and sometimes 
 this inequality is sharp -- if all the graphs are complete graphs -- 
 but in general it is not; for example, $R_{\rm hom}(C_3,C_5)=5$ 
 although $R(C_3,C_5)=9$. It may be of independent interest to 
 investigate the properties of $R_{\rm hom}$ further.

 Given a graph $G$, a vertex $x$ of $G$, and an integer $m\geq 1$, the 
 $m$-\textit{blow-up} of $G$ at $x$ is the graph obtained by replacing 
 $x$ with an independent set $\{x_1,\ldots,x_m\}$, each vertex of which 
 has the same adjacencies as $x$ has in $G$. The $m$-\textit{blow-up} 
 of the graph $G$ is the graph obtained by blowing up by $m$ at each 
 vertex of~$G$. It is clear that there is a homomorphism of $H_i$ into 
 the $i$th colour subgraph of $G$ if and only if there is some 
 $m$-blow-up of $G$ whose $i$th colour subgraph contains $H$. Hence 
 $R_{hom}$ can be also defined in terms of blow-ups.

 We define $Z$ as the smallest natural number $N$ with the following  
 property. For any labelled graph $F$ on $W+N$ vertices, with all  
 edges incident to at least one of the first $W$ vertices present, if 
 $F$ is $r$-coloured and $F'$ is obtained from $F$ as the $m$-blow-up 
 of the first $W$ vertices for some sufficiently large 
 $m=m(H_1,\ldots,H_r)$,
 then there is some $i$ such that $H_i$ is contained within the
 $i$th colour subgraph of $F'$.

 It is clear that, for any connected $n$-vertex graph $G$ with 
 $n\geq Z$, $R(G,H_1,\ldots,H_r)\geq W(n-1)+Z$. Indeed, let $F$ be a 
 labelled $r$-coloured graph on $W+Z-1$ vertices demonstrating that $Z$ 
 cannot be replaced by $Z-1$. We obtain an $r$-coloured complete graph 
 $F'$ on $W(n-1)+Z-1$ vertices by taking the $(n-1)$-blow-up of the first $W$ 
 vertices of $F$ and then replacing every non-edge with a red edge. 
 Then certainly there is no red copy of $G$ in $F'$, and, by the 
 definition of $Z$, for each $i\in [r]$, there is no copy of $H_i$ 
 in~$F'$.

We now prove Theorem~\ref{MAGisAG}.

\begin{proof} 
 Given an always-good class of graphs $\mathcal{G}$, let $r\geq 2$ be 
 any integer and $H_1,\ldots,H_r$ any collection of $r$ graphs. Let 
 $W$, $Z$, and $m=m(H_1,\ldots,H_r)$ be defined as above.

 Suppose that $\ell$ is some integer sufficiently large that the 
 following procedure succeeds.

 Let $F$ denote a complete $(W+1)$-partite graph 
 $K_{\ell, \ell,\ldots, \ell,Z}$, with an $r$-colouring of its edges. Let 
 $(U_1,V_1), \ldots, (U_{\binom{W}{2}},V_{\binom{W}{2}})$ be an 
 enumeration of all the pairs of parts of $F$, excluding the last part. 
 Set $\Gamma_0$ equal to the set of vertices in these first $W$ parts.

 Now, for each $1\leq i\leq\binom{W}{2}$ in turn, consider the complete 
 bipartite $r$-coloured subgraph $J_i$ of $F$ induced by the pair 
 $(U_i \cap \Gamma_{i-1}, V_i \cap \Gamma_{i-1})$. Let $B_i$ be a 
 maximum-size monochromatic complete balanced bipartite subgraph of 
 $J_i$, and form $\Gamma_i$ by deleting from $\Gamma_{i-1}$ all 
 vertices in $U_i \cap \Gamma_{i-1}$ and $V_i \cap \Gamma_{i-1}$ except 
 those in $B_i$.

 Observe that, at each step $i$, the $r$-coloured complete bipartite 
 graph $J_i$ 
 must have one colour present with edge density at least $1/r$, and 
 thus the K\"ov\'ari-S\'os-Tur\'an Theorem (Theorem~\ref{KSTBound}) 
 provides a lower bound on the size of the monochromatic complete 
 balanced bipartite subgraph $B_i$ found at step $i$. In particular, by 
 choosing $\ell$ sufficiently large, we may conclude that, at the end 
 of the process, the set $U \cap \Gamma_{\binom{W}{2}}$ contains at 
 least $r^Zm$ vertices for every part $U$ of $F$, except the last one.

 Now let $F'$ be the $r$-coloured $(W+1)$-partite graph obtained from 
 $F$ by removing all vertices $\Gamma_0-\Gamma_{\binom{W}{2}}$ -- in 
 other words, we take the $r$-coloured complete $W$-partite graph 
 induced on $\Gamma_{\binom{W}{2}}$ and add back the last part of $F$.  
 By construction, the edges between any two of the first $W$ parts of 
 $F'$ form a monochromatic complete bipartite graph.  Let $F''$ be 
 obtained by deleting from the first $W$ parts of $F'$ a minimum set of 
 vertices such that the edges from any vertex in the $(W+1)$-st part of 
 $F''$ to any of the first $W$ parts are monochromatic.  By choice of  
 $\ell$, the first $W$ parts of $F$ each still contain at least $m$ 
 vertices. 
 By the definition of $W$ and $Z$, for some $i$, a copy of~$H_i$ of 
 colour $i$ is contained in $F''$, and, hence, in $F$.

 Now, because $\mathcal{G}$ is always-good, in particular it is  
 $H$-good for $H = K_{\ell,\ell,\ldots,\ell,Z}$. Note that 
 $\chi(H) = W+1$ and $\sigma(H) = Z$. Thus, whenever $G\in\mathcal{G}$ 
 is sufficiently large, and $F$ is any $\{$red$\}\cup [r]$-coloured 
 complete graph on $W(|G|-1)+Z$ vertices, either $F$ contains a red 
 copy of $G$, or $F$ contains an $r$-coloured copy of 
 $H = K_{\ell,\ell,\ldots,\ell, Z}$, and hence a copy of $H_i$ of 
 colour $i$ for some $i\in [r]$.
\end{proof}

In the proof above, the required size of $\ell$ is a tower of height $O(W^2)$, and in turn $W$ can be very large in comparison to the small graphs---for instance, if each of the small graphs is the clique $K_s$, then $W=2^{\Omega(s)}$.

As an illustration of the use of Theorem~\ref{MAGisAG}, we show how to find the Ramsey numbers for a collection of
odd cycles, provided they are suitably long.

\begin{corollary}  For any odd integers $\ell_1,\ldots,\ell_r$, with $\ell_s> 2^s$ for each $1\le s\le r$, and
every sufficiently large $n$,
$$
 R(C_n,C_{\ell_1},\ldots,C_{\ell_r})=2^r(n-1)+1.
$$
\end{corollary}

\begin{proof}
 Since the family of cycles is always-good by 
 Theorem~\ref{2alwaysgood}, and thus by Theorem~\ref{MAGisAG} 
 multicolour-always-good, we need only solve the two auxiliary 
 Ramsey-type problems to find $W$ and $Z$.

 We first need to show that 
 $R_{\rm hom}(C_{\ell_1},\ldots,C_{\ell_r}) =2^r+1$. For the lower 
 bound, it is a standard result that the edge-set of $K_{2^r}$ can be 
 partitioned into $r$ colour classes so that no colour class contains 
 an odd cycle.   
 For such a colouring, there is no homomorphism from the odd cycle 
 $C_{\ell_i}$ to the $i$th colour class, for any $i$.  

 For the upper bound, we proceed by induction on $r$. The result is 
 trivial for $r=1$, so we suppose $r>1$. For any $r$-coloured  
 $K_{2^r+1}$, either there is an odd cycle $Q$ in colour $r$, in which 
 case there is a homomorphism from $C_{\ell_r}$ to $Q$ -- here we use 
 the assumption that $\ell_r \ge 2^r +1$ -- or the graph of edges 
 coloured by $r$ is bipartite, in which case one of its parts contains 
 an $(r-1)$-coloured $K_{2^{r-1}+1}$, and the result follows by 
 induction.

 Secondly, we need to show that $Z=1$. Any labelled graph $F$ on 
 $W+Z=R_{\rm hom}(C_{\ell_1},\ldots,C_{\ell_r})$ vertices, with all  
 edges incident to at least one of the first $W$ vertices present, is 
 complete. Hence, if $F$ is $r$-coloured, then, for some $i$, there is 
 a homomorphism from $C_{\ell_i}$ into the $i$th colour subgraph of $F$.
 Thus, $F$ contains an odd cycle $Q$ of length at most $\ell_i$ in 
 colour $i$.

 For $m= \max(\ell_1,\ldots,\ell_r)$, let $F'$ be obtained from $F$ as 
 the $m$-blow-up of the first $W$ vertices. Given an odd cycle $Q$ 
 in colour $i$, we have freedom to choose a homomorphism which maps 
 only one vertex of $C_{\ell_i}$ to a chosen vertex of $Q$. By 
 $m$-blowing-up the remaining vertices of $Q$, we obtain enough room to 
 embed the remaining vertices of $C_{\ell_i}$.
%
\end{proof}

 A well-known problem raised by Bondy and Erd\H os~\cite{BonEr} is to 
 determine the $r$-colour Ramsey number $R(C_n,\dots,C_n)$, when $n$ is 
 odd. The lower bound $2^{r-1}(n-1) +1$ is pointed out in that paper, 
 and this is widely believed to give the correct value of the Ramsey 
 number provided $n$ is sufficiently large -- our result above may be 
 seen as giving a weak support for that conjecture. The conjecture was 
 proved in the case $r=3$ by Kohayakawa, Simonovits and  
 Skokan~\cite{KSS3cyc}: for $n$ odd and sufficiently large, 
 $R(C_n,C_n,C_n) = 4n-3$.

\section{Powers of paths and cycles against themselves}\label{SecPnkCnk}

Our purpose in this section is to give both upper and lower bounds on the Ramsey numbers
$R(P^k_n,P^k_n)$ and $R(C^k_n,C^k_n)$, for fixed $k$ and large $n$.  In the next section, we will move on to 
consider general graphs with bounded maximum degree and limited bandwidth.  

We start with a construction giving a lower bound better than the one from Burr's construction (Lemma~\ref{LowBd}).  
We begin by assuming that $n$ is a multiple of $k+1$: for convenience we restate Theorem~\ref{LowerBound}.


\begin{reptheorem}{LowerBound}
 For $k\geq 2$,
 \[
   R(C^k_{(k+1)t}, C^k_{(k+1)t}),\, R(P^k_{(k+1)t}, P^k_{(k+1)t})
    \geq
   t(k+1)^2-2k~.
  \]
\end{reptheorem}

 Note that $\chi(C^k_{(k+1)t}) = \chi(P^k_{(k+1)t}) = k+1$, while
 $\sigma(C^k_{(k+1)t}) = \sigma(P^k_{(k+1)t}) =t$, so Lemma~\ref{LowBd} 
 gives the lower bound $k[(k+1)t - 1] + t$ on both Ramsey numbers, 
 which is $kt-k$ below the value in Theorem~\ref{LowerBound}.

\begin{proof}
 We colour $K_{t(k+1)^2-2k-1}$ as follows. Partition $[t(k+1)^2-2k-1]$ 
 into disjoint sets $A_1,\ldots,A_k$ each on $kt-1$ vertices, 
 $B_1,\ldots,B_k$ each on $2t-1$ vertices, and $C$ on $t-1$ vertices.

 Now colour edges as follows. Within each set $A_i$ we have only red 
 edges. Within each set $B_i$ we have only blue edges. Between two sets 
 $A_i$ and $A_j$, $i\neq j$, we have only blue edges; between $B_i$ and 
 $B_j$, $i\neq j$, only red edges.

 For each $i$, we have only red edges between $A_i$ and $B_i$, while 
 between $A_i$ and $B_j$ for $i\neq j$ we have only blue edges. 
 Finally, we take any colouring within $C$, and join all its vertices 
 in blue to every $A_i$ and in red to every $B_i$.

 In the red graph, any copy of $P^k_m$, for any $m>k$, with one vertex 
 in a set $A_i$ must lie entirely within $A_i\cup B_i$, $A_i$ is too 
 small to contain $k$ colour classes of $P^k_{(k+1)t}$, hence, $B_i$ 
 must contain two vertices from two distinct vertex classes, which is 
 impossible because all its edges are blue. But if the sets $A_i$ are
 not to be used, then an entire colour class of $P^k_{(k+1)t}$ would 
have to lie in $C$, which is again too small.

The argument showing that there is no copy of $P^k_{(k+1)t}$ in the blue graph is very similar.  Suppose there
is such a copy $Q$, and suppose first that it includes some vertex $v$ of some $B_i$.  The set $T$ of the next
$k$ vertices on $Q$ forms a blue clique adjacent to $v$, so there is at most one vertex of $T$ in each of the $A_j$
with $j\neq i$, and so at least one vertex of $T$ in $B_i$.  Thus $Q$ lies within $B_i \cup \bigcup_{j\neq i} A_j$.
Moreover, at most $k-1$ out of each set of $k+1$ consecutive vertices on $Q$ are in the $A_j$, and so there
are at least $2t$ vertices of $Q$ in $B_i$, which is impossible.  As before, if the sets $B_i$ are not
used for $Q$, then an entire colour class of $Q$ would lie in $C$, which is again too small.
\end{proof}

 For $k>2$, the construction above can be generalised.  First we take 
 an auxiliary red-blue-coloured graph $J$, which is a copy of 
 $K_{k,k}$, with parts $\{a_1, \dots, a_k\}$ and $\{b_1, \dots, b_k\}$, 
 with the property that each $a_i$ is incident with at least one blue 
 edge and each $b_i$ with at least one red edge. Each such $J$ will 
 give us a different construction of a two-coloured graph on 
 $t(k+1)^2-2k-1$ vertices with no monochromatic $P^k_{(k+1)t}$, as 
 follows. Take disjoint vertex sets $A_1,\ldots, A_k, B_1,\ldots, B_k, 
 C$, with $|C|=t-1$. The set $A_i$ has $(\ell_i+1)t - 1$ vertices, 
 where $\ell_i$ is the number of blue edges incident with $a_i$ in $J$, 
 and the set $B_i$ has $(m_i+1)t-1$ vertices, where $m_i$ is the number 
 of red edges incident with $b_i$ in $J$.  The total number of vertices 
 is always $t(k+1)^2-2k-1$. 

 As before, within each $A_i$ we have red edges, and between different 
 $A_i$ we have blue edges, while within each $B_i$ we have blue edges, 
 and between different $B_i$ we have red edges. The colouring inside 
 $C$ is arbitrary, and its vertices are joined in blue to every $A_i$ 
 and in red to every $B_i$. The edges between $A_i$ and $B_j$ all have 
 the colour of the edge between $a_i$ and $b_j$ in~$J$. The proof that 
 such a two-coloured graph contains no monochromatic $P^k_{(k+1)t}$ is 
 similar to that in Theorem~\ref{LowerBound}.

 When $k+1$ does not divide $n$, still Burr's construction can be 
 improved upon. For powers of paths, the adjustments required are 
 small, so we concentrate on powers of cycles.

 \begin{theorem}
  For $k\geq 2$ and any $1\leq r\leq k$,
  \[
    R(C^k_{(k+1)t+r},C^k_{(k+1)t+r})\geq (k+1)[(k+2)t+2r-2]+r~.
  \]
\end{theorem}

 The lower bound on $R(C^k_{(k+1)t+r},C^k_{(k+1)t+r})$ coming from 
 Lemma~\ref{LowBd} is $(k+1)[(k+1)t + r - 1] + r$.

 \begin{proof}
  We colour $K_{(k+1)((k+2)t+2r-2)+r-1}$ as follows. Partition  
  $[(k+1)((k+2)t+2r-2)+r-1]$ into disjoint sets $A_1,\ldots,A_{k+1}$ 
  each on $kt+r-1$ vertices, $B_1,\ldots,B_{k+1}$ each on $2t+r-1$ 
  vertices, and $C$ on $r-1$ vertices.

  Now we colour edges as in the proof of the previous theorem. The 
  proof that such a two-coloured graph contains no monochromatic 
  $C^k_{(k+1)t+r}$ follows the proof of Theorem~\ref{LowerBound} and we 
  omit it here.
\end{proof}

We believe that the constructions above are, at least asymptotically, optimal.

The rest of this section is devoted to proving the upper bounds on $R(P_n^k,P_n^k)$ and $R(C_n^k,C_n^k)$
given in Theorem~\ref{RPnkCnk}.  Our first step is to prove an upper bound on $R(P_n,P_n^k)$, for which we 
need the following three results.

First, we recall the Erd\H{o}s-Gallai extremal theorem for cycles~\cite{ErdGall}.

\begin{theorem} [Erd\H os-Gallai~\cite{ErdGall}] \label{ErdosGallai}
Let $G$ be a graph on $n$ vertices, and $c$ an integer, $3\leq c\leq n$.
Then either $G$ contains a cycle of length at least $c$ or
\[e(G)<(c-1)(n-1)/2+1~.\]
\end{theorem}

Second, we need a result on maximum cycles in graphs.  The lemma below is simple, and convenient for our
purposes: a much stronger result has recently been proved by Kohayakawa, Simonovits and Skokan~\cite{KSS}

\begin{lemma} \label{MaxCycle} Given a graph $G$ containing vertex disjoint cycles $C_t$ and $C_{t'}$, if $G$
contains no cycle of length greater than $t$, then the bipartite graph $G[V(C_t),V(C_{t'})]$ contains no copy
of $K_{s,s}$, where $s=\left\lceil\frac{t}{t'}\right\rceil+2$.
\end{lemma}

\begin{proof}
Suppose not, and let $G$, $C_t$, $C_{t'}$ form a counterexample.  Now $G$ contains a copy of the bipartite
graph $K_{s,s}$ whose parts are in $V(C_t)$ and $V(C_{t'})$, so in particular there are two vertices of this
complete bipartite graph in $C_t$ which are joined in $C_t$ by a path $P$ of length at least $\frac{s-1}{s}t$,
and two more in $C_{t'}$ joined by a path $P'$ in $C_{t'}$ of length at least $\frac{s-1}{s}t'$.
The vertices in $V(P)\cup V(P')$ form a cycle of length at least $\frac{s-1}{s}(t+t')>t$, which is a
contradiction.
\end{proof}

 Third, a standard greedy method allows us to find a copy of $P_n^k$ in 
 a very dense graph on only slightly more than $n$ vertices.

\begin{lemma} \label{CoverW}
 Let $k$ and $n$ be natural numbers, and $\varepsilon$ a real number, 
 satisfying $0<\varepsilon\leq (k+3)^{-1}$ and $n>3\varepsilon^{-2}$.  
 If $H$ is any graph on at least $n+(k+2)\varepsilon n$ vertices, such 
 that the complement $\overline{H}$ contains no cycle of length
 at least $\varepsilon^2 n$, then $H$ contains a copy of $P_n^k$.
\end{lemma}


\begin{proof}
 By Theorem~\ref{ErdosGallai}, $\overline{H}$ has at most 
 $(\varepsilon^2 n-1)(|H|-1)/2+1<\varepsilon^2 |H| n/2$ edges. If 
 $\overline{H}$ had more than $n+k\varepsilon n$ vertices of degree 
 greater than $\varepsilon n$, then it would have at least 
 $(n+k\varepsilon n)\frac{\varepsilon n}{2}$ edges, which is a 
 contradiction. So at least $n+k\varepsilon n$ vertices of 
 $\overline{H}$ have degree less than $\varepsilon n$.   
 Let $H'$ be the subgraph of $H$ induced by these vertices. Then $H'$ 
 has at least $n+k\varepsilon n$ vertices, the neighbourhood of any
 set of $k$ vertices of $\overline{H'}$ contains at most 
 $k\varepsilon n$ vertices, and so, in $H'$, every set of $k$
 vertices has at least $n$ common neighbours. We can embed $P_n^k$ into 
 $H'$ by a simple greedy procedure: we choose any vertex to be the 
 first vertex of the path, any neighbour to be the second vertex of the 
 path, and so on. At each embedding step, we only need to find a vertex 
 which is adjacent to all of the last $k$ vertices embedded, and which 
 has not yet been used in the embedding. Such a vertex is guaranteed to 
 exist since any $k$ vertices of $H'$ have at least $n$ common 
 neighbours, and we only need to embed a total of $n$ vertices.
\end{proof}

Now we can prove our upper bound on the Ramsey number of a path versus 
a power of a path.

\begin{lemma} \label{RamseyPkP}
 For any natural number $k$,
 \[
  R(P_n,P_n^k) \le \left(k+1+\frac{1}{k+1}\right)n+o(n)~.
 \]
\end{lemma}

 Note that this upper bound is significantly larger than the lower  
 bound $R(P_n,P_n^k)\geq k(n-1)+\sigma(P_n^k)\sim 
 \left(k+\frac{1}{k+1}\right)n$. We conjecture that the lower bound is 
 correct. An improvement in this upper bound would improve the upper 
 bound in Theorem~\ref{RPnkCnk} by a corresponding amount, but this is 
 not the source of the factor of~2 between our lower and upper bounds. 

\begin{proof}
 We show that, for any $0<\varepsilon\leq (k+3)^{-1}$, the Ramsey 
 number $R(P_n,P_n^k)$ is bounded above by
 \begin{gather}
   \left(k+1+\frac{1}{k+1}+(k+3)\varepsilon\right)n\label{as:1}\\
   \intertext{ for }
   n>\left(16(2k+1)\varepsilon^{-8}\right)^{4\varepsilon^{-2}}~. 
    \label{as:2}
\end{gather}

 Accordingly, we assume that $n$ is indeed greater than  
 $\left(16(2k+1)\varepsilon^{-8}\right)^{4\varepsilon^{-2}}$.
 Let $G$ be a two-edge-coloured complete graph on  
 $\left(k+1+\frac{1}{k+1}+(k+3)\varepsilon\right)n$ vertices which 
 contains no red $P_n$.  We choose successively vertex-disjoint 
 maximum-length red cycles in $G$. Let $V_1$ be the vertex set of the 
 longest red cycle of $G$, $V_2$ the vertex set of the longest red 
 cycle of $G-V_1$, and so on.

 Since $P_n\subset C_n$, we have $n-1\geq|V_1|\geq|V_2|\geq\cdots$.  
 Let $r$ be the greatest index such that $|V_r|\geq \varepsilon^2 n$, 
 and let $W=V(G)-\displaystyle\bigcup_{i=1}^rV_i$.  Since the sets 
 $V_i$ are disjoint, we have 
 $r\leq (k+1+\frac{1}{k+1}+(k+3)\varepsilon)\varepsilon^{-2} 
  <\varepsilon^{-3}$, independently of $n$.

 If $|W|\geq n+(k+2)\varepsilon n$, then the graph of blue edges in $W$ 
 satisfies the conditions of Lemma~\ref{CoverW}, so $G$ contains a blue 
 copy of $P_n^k$. Therefore we will assume $|W|<n+(k+2)\varepsilon n$.

 Let $s=\left\lceil\frac{n}{\varepsilon^2 n}\right\rceil+2 < 
 2\varepsilon^{-2}$. By Lemma~\ref{MaxCycle}, for any
 $1\leq i<j\leq r$, there is no red copy of $K_{s,s}$ in $G$ whose 
 parts are in $V_i$ and $V_j$ respectively. We wish to use this 
 together with Lemma~\ref{MyBlowUp} to find a blue copy of $P_n^k$  
 (which has maximum degree $2k$). We will use the fact that $P_n^k$ is 
 a subgraph of the complete $(k+1)$-partite graph with parts of size
 $\left\lceil\frac{n}{k+1}\right\rceil$. Observe that no part $V_i$ has 
 size greater than $n$, and the union of all the parts has size at 
 least $\left(k+\frac{1}{k+1}+\varepsilon\right)n$.

 Now choose $\ell_1$ to be the smallest index such that
 \[
  \sum_{i=1}^{\ell_1} \left(|V_i|-4s^2n^\frac{2s-1}{2s}(2k+1)\right) 
   \geq
  \left\lceil\frac{n}{k+1}\right\rceil~.
 \]
 Since $4s^2n^\frac{2s-1}{2s}(2k+1)r<\varepsilon n$ (here we use 
 \eqref{as:2}), by \eqref{as:1} this is possible and, furthermore, 
 $\displaystyle\sum_{i=1}^{\ell_1}|V_i|<n$ (in fact, this sum can 
 exceed $2\left\lceil\frac{n}{k+1}\right\rceil +\varepsilon n$ only 
 when $\ell_1=1$).

For each $2\leq j\leq k$ in succession, let $\ell_j$ be the smallest 
index such that
\[
 \sum_{i=\ell_{j-1}+1}^{\ell_j}   
   \left(|V_i|-4s^2n^\frac{2s-1}{2s}(2k+1)\right)
 \geq
 \left\lceil\frac{n}{k+1}\right\rceil~.
\]
 Again, this is possible because $\displaystyle \sum_{i=1}^{\ell_{j-1}} |V_i|<(j-1)n$ 
 and $4s^2n^\frac{2s-1}{2s}(2k+1)r<\varepsilon n$, and we also have 
 $\displaystyle \sum_{i=\ell_{j-1}+1}^{\ell_j}|V_i|<n$.

 We apply Lemma~\ref{MyBlowUp} to the parts $V_1,\ldots,V_r$ of $G$.  
 Let $V'_1,\ldots,V'_r$ be the parts of $G'$ as in the lemma; since for 
 each $1\leq i<j\leq r$ the sets $V_i$ and $V_j$ are blue-adjacent, the 
 parts $V'_i$ and $V'_j$ span a complete bipartite graph. Let 
 $W_1=\displaystyle\bigcup_{i=1}^{\ell_1}V'_i$,
 $W_j=\displaystyle\bigcup_{i=\ell_{j-1}+1}^{\ell_j}V'_i$ for each 
 $2\leq j\leq k$, and 
 $W_{k+1}=\displaystyle\bigcup_{i=\ell_k+1}^r V'_i$. 
 Since $|W_1|,\dots, |W_k|<n$ and \eqref{as:1} holds, we are guaranteed 
 to find that $|W_{k+1}|\geq\left\lceil\frac{n}{k+1}\right\rceil$. The 
 $W_j$ form the parts of a complete $(k+1)$-partite subgraph of $G'$, 
 so that $P_n^k$ can be embedded into $G'$.  By Lemma~\ref{MyBlowUp}, 
 $G$ contains a blue copy of $P_n^k$.
\end{proof}

It is now straightforward to prove our desired bounds on $R(P_n^k,P_n^k)$ and $R(C^k_n,C^k_n)$.

\begin{proof}[Proof of Theorem~\ref{RPnkCnk}] Given $\varepsilon>0$, let $s=s(n)$ be any sufficiently slowly
growing function of $n$ and $n_0$ be any sufficiently large integer. Suppose $n>n_0$.

Suppose first that we seek either a monochromatic $P^k_n$, or a monochromatic $C^k_n$ where $k+1$ divides $n$. 
Let $G$ be a two-coloured complete $N$-vertex graph, where
\[N=\left(2k+2+\frac{2}{k+1}\right)n+\varepsilon n~.\]

We partition $V(G)$ into a collection $\mathcal{R}$ of red $s$-cliques, $\mathcal{B}$ of blue $s$-cliques,
and a leftover set of at most $2^{2s}$ vertices.  Without loss of generality, we assume that
$|\mathcal{R}|\geq|\mathcal{B}|$.

We call two cliques $R_i$ and $R_j\in\mathcal{R}$ red-adjacent when 
$G[R_i,R_j]$ contains a copy of $K_{4k,4k}$, and blue-adjacent 
otherwise.  
This defines the two-coloured complete graph $G^*$ on $\mathcal{R}$. 
For $t=n/s+\varepsilon n/4s(k+2)$, we have
\[|
  G^*|\geq \frac{N-2^{2s}}{2s}\geq \frac{N-\varepsilon n/2}{2s}
      \geq \left(k+1+\frac{1}{k+1}\right)t+\frac{\varepsilon t}{8}.
\]
If we find a red copy of $P_t$ in $G^*$, then we immediately find a red 
copy both of $P^k_n$ and of $C^k_n$ in $G$, as $st\geq n$.

But by Lemma~\ref{RamseyPkP}, if we do not have in $G^*$ a red copy of 
$P_t$, then we do have a blue copy of $P^k_t$; by Lemma~\ref{MyBlowUp}, 
we find in $G$ a copy of $P^k_n$ and, provided that $k+1$ divides $n$, 
also of $C^k_n$, as required.

If we seek a monochromatic copy of $C^k_n$ and $k+1$ does not divide 
$n$, then observe that (provided $n>(k+1)^2$) we have 
$\chi(C^k_n)=k+2$.  We use the same strategy, now applying 
Lemma~\ref{RamseyPkP} to find either a red $P_t$ or blue $P^{k+1}_t$ in 
the (by assumption larger) graph $G^*$, to obtain the desired result.
\end{proof}

We note that the primary reason why the upper bound we obtain is larger than the
conjectured value by approximately a factor of~2 is that, in this proof, we simply throw away the 
minority colour cliques.

\section{Ramsey numbers of poor expanders}

In this section we prove Theorem~\ref{BBwRam}. We prove this theorem by combining
Theorem~\ref{BSTThm} with a variation of Theorem~\ref{RPnkCnk}.

\begin{lemma}\label{SzPartPnk} Given sufficiently small $\varepsilon>0$ and
integer $k$, there exists $n_0$ such that the following holds.  Let $n\geq n_0$
and $G$ be any three-edge-colouring of the complete graph $K_{(2k+3)n}$, with
edges coloured either `red', `blue', or `bad', such that no more than
$\varepsilon n$ bad edges meet any single vertex.  Then $G$ contains either a red
or a blue copy of $P^k_n$.
\end{lemma}


The proof of this lemma is a straightforward modification of the
proof of Theorem~\ref{RPnkCnk}, in much the same way as
Lemma~\ref{StabPnkH:modified} is a straightforward modification of
Lemma~\ref{StabLem2}. As there, we must replace the Erd\H{o}s-Szekeres bound
with an easy modification to find red and blue cliques, and, as there, we must
permit our auxiliary graph $G^*$ to contain some non-edges (but not too many at any
vertex). It is straightforward to check that the remainder of the proof is
insensitive to this change; we omit the details.

 We are now in a position to complete the proof of Theorem~\ref{BBwRam}.

\begin{proof}[Proof of Theorem~\ref{BBwRam}] 
Given $\Delta$, let $\mu=1/(30(\Delta+1)^2)$ and $\gamma=1/2$. 
For $k$, $2\leq k\leq \Delta+1$, let $\beta_{\subref{BSTThm}}=\beta_{\subref{BSTThm}}(k),\eps_{\subref{BSTThm}}>0,$ and $n_{\subref{BSTThm}}=n_{\subref{BSTThm}}(k)$ be constants such that
whenever $n\geq n_{\subref{BSTThm}}$, Theorem~\ref{BSTThm} permits the embedding into a
$(2k+4)n$-vertex graph $F$ (possessing a suitable $\eps_{\subref{BSTThm}}$-regular partition)
of any $n$-vertex graph $G$ with $\Delta(G)\leq\Delta$, $\chi(G)=k$, and $\bw(G)\leq \beta_{\subref{BSTThm}} n$.

We set $\beta=\min\{\beta_{\subref{BSTThm}}(k), 2\leq k\leq \Delta+1\}$ and $n_0=\max\{n_{\subref{BSTThm}}(k), 2\leq k\leq \Delta+1\}$.
Let $G$ be any $n$-vertex graph with $\Delta(G)\leq\Delta$ and $\bw(G)\leq\beta n$. Set $k=\chi(G)\leq\Delta+1$, and let $F$ be any complete $2$-coloured graph on $(2k+4)n$
vertices.

By Theorem~\ref{SzLem}, $F$ possesses an $\eps_{\subref{BSTThm}}$-regular partition. Let $R(F)$
be the corresponding $(2k+3)m$-vertex cluster graph, with edges coloured `red'
when they correspond to $\eps_{\subref{BSTThm}}$-regular pairs whose density of red edges is at
least $\frac{1}{2}$, `blue' when they correspond to $\eps_{\subref{BSTThm}}$-regular pairs whose
density of red edges is less than $\frac{1}{2}$, and `bad' otherwise.

By Lemma~\ref{SzPartPnk} applied to $R(F)$, $R(F)$ contains either a red or a
blue copy of $P^k_m$. By symmetry we may presume that it is a red copy.

Observe that $\frac{1}{2k+4}+\mu\leq\frac{1}{2k+3}$. It follows that we may set
$\eta=\frac{1}{2k+4}$ and apply Theorem~\ref{BSTThm} to the graph formed by the
red edges of $F$, with the $\eps_{\subref{BSTThm}}$-regular partition given, to find a copy of
$G$; this is a red copy of $G$ in $F$, completing the proof. \end{proof}

We made no effort to optimise the constants implicit in either Lemma~\ref{SzPartPnk} or
Theorem~\ref{BBwRam}. It seems very likely that given any $\eps>0$ there is
$\delta>0$ such that the following is true for sufficiently large $n$. If $F$ is
any two-coloured complete graph on $R(P^k_n,P^k_n)+\eps n$ vertices, then even
after deleting $\delta n$ edges meeting each vertex of $F$, there remains either
a red or a blue copy of $P^k_n$ in $F$. It would then follow that given $\eps>0$,
if $G$ is any $n$-vertex graph with maximum degree $\Delta$, chromatic number $k$
and bandwidth $\beta n$, where $\beta$ is sufficiently small and $n$ sufficiently
large, then $R(G,G)\leq R(P^k_n,P^k_n)+\eps n$.

It seems likely that $R(G,G)\leq R(P^{k-1}_n,P^{k-1}_n)+\eps n$ is true. However
to prove this (at least by the methods used here) one would need to be able to
find in the Szemer\'edi cluster graph not only a monochromatic $(k-1)$st power of a path of
sufficient length, but also a structure (for example an appropriately positioned
$(k+1)$-clique in the same colour) permitting redistribution of vertices between
colour classes.

\section{Open Problems}

We collect here some open problems related to our work, including some that we have mentioned in the paper.

First, we wonder whether there is scope for some improvement in Theorem~\ref{NExpAlwaysGood}: can we weaken the hypothesis
that the bandwidth be sublinear?

\begin{problem}
Is there, for any $d\geq 3$, a constant $\eps_d > 0$ such that the class ${\mathcal G}_{d,\eps_d n}$ of graphs $G$ with maximum degree~$d$ and
bandwidth at most $\eps_d |G|$ is always-good?
\end{problem}

Another possibility for weakening the hypotheses of Theorem~\ref{NExpAlwaysGood} is to replace the bound
on the maximum degree by a bound on the degeneracy of $G$.

\begin{conjecture}
For each fixed $d$, and each function $\beta (n)=o(n)$, the class ${\mathcal G}'_{d,\beta}$ of graphs $G$ with degeneracy
at most $d$ and bandwidth at most $\beta(|G|)$ is always-good.
\end{conjecture}

We discussed earlier the need for $n$ to be quite large in terms of $|H|$ in order for $R(P_n^k,H)$ to be as small as
$(\chi(H)-1)(n-1) + \sigma(H)$, for $k \ge 2$; we also mentioned the following conjecture.

\begin{conjecture}
For every graph $H$,
$R(P_n,H) = (\chi(H)-1)(n-1) + \sigma(H)$ whenever $n \ge |H|$.
\end{conjecture}

We believe that our lower bound on $R(P_n^k,P_n^k)$ is in fact the correct value for this Ramsey number.
We state this conjecture for $n$ a multiple of $k+1$, for convenience, but we believe that our construction in
Section~\ref{SecPnkCnk} is optimal for all sufficiently large values of $n$.

\begin{conjecture} For $k \ge 2$, and $n$ a sufficiently large multiple of $k+1$, we have
$$
R(P^k_n,P^k_n) = (k+1)n-2k.
$$
\end{conjecture}

We believe that the same result is also true for $C^k_n$.

A proof of the above conjecture would give some improvement in the 
bound in Theorem~\ref{BBwRam}.  As mentioned in the introduction, we 
expect the following to be true.

\begin{conjecture}
 For each $\Delta \ge 1$, there exist $n_0$, $\beta$  and $C$ such 
 that, whenever $n \ge n_0$ and $H$ is an $n$-vertex graph with maximum 
 degree at most $\Delta$ and bandwidth at most $\beta n$, we have 
 $R(H,H) \le (\chi(H)+C)n$.
\end{conjecture}

As discussed at the end of the previous section, we may be able to take 
$C$ to be arbitrarily small.

The graph $P^3_n$ is easily seen to be planar for every $n$; by
Theorem~\ref{LowerBound} we have $R(P^3_n,P^3_n)\geq 4n-6$ when $4$ divides
$n$. We know of no planar graphs with larger Ramsey number, but we have not
made any serious efforts to discover such. 
Chen and Schelp proved~\cite{ChenSchelp} that there exists an absolute 
constant $C$ such that $R(H,H)\leq Cn$ for every $n$-vertex planar 
graph $H$. The best value known to us for $C$ is obtained by combining a 
theorem of Graham, R\"odl and Ruci\'nski~\cite{GRRLinRam} (essentially Theorem~\ref{GRR}) with the
Kierstead-Trotter bound~\cite{KieTrot} that all planar graphs are $10$-arrangeable, which yields
$C\approx 10^{200}$. By Corollary~\ref{PlanarCor} we can reduce $C$ to $12$
for bounded degree planar graphs. We offer the following conjecture.

\begin{conjecture} For every sufficiently large $n$ and every planar graph $H$
on $n$ vertices, we have $R(H,H)\leq 12n$.
\end{conjecture}

We know of no example of a planar graph $H$ on $n$ vertices for which $R(H,H)>4n+2$: $H=K_4$ and $H=K_5 - e$ attain this bound (see \cite{Radz}).



\end{document}